\documentclass[11pt]{amsart}

\usepackage[T1]{fontenc}
\usepackage[utf8]{inputenc}
\usepackage{lmodern}
\usepackage{amsmath,amssymb,amsthm,mathtools}
\usepackage{enumitem}
\usepackage{geometry}
\usepackage{microtype}
\usepackage[colorlinks=true,linkcolor=blue,citecolor=blue,urlcolor=blue]{hyperref}

\newtheorem{theorem}{Theorem}[section]
\newtheorem{proposition}[theorem]{Proposition}
\newtheorem{lemma}[theorem]{Lemma}
\newtheorem{corollary}[theorem]{Corollary}
\newtheorem{definition}[theorem]{Definition}
\newtheorem{remark}[theorem]{Remark}

\numberwithin{equation}{section}

\newcommand{\N}{\mathbb N}
\newcommand{\Nzero}{\mathbb N_0}
\newcommand{\Z}{\mathbb Z}
\newcommand{\R}{\mathbb R}
\newcommand{\C}{\mathbb C}
\newcommand{\supp}{\operatorname{supp}}
\newcommand{\Aut}{\operatorname{Aut}}
\newcommand{\Fix}{\operatorname{Fix}}
\newcommand{\ord}{\operatorname{ord}}

\newcommand{\Cat}{\mathsf{Cat}}
\newcommand{\Ehr}{\operatorname{Ehr}}

\newcommand{\M}{\mathsf M}
\newcommand{\calF}{\mathcal F}
\newcommand{\calE}{\mathcal E}
\newcommand{\calP}{\mathcal P}
\newcommand{\calO}{\mathcal O}

\newcommand{\bdot}{\boldsymbol{\cdot}}
\newcommand{\qbinom}[2]{\genfrac{[}{]}{0pt}{}{#1}{#2}_{q}}

\title[Zero-sum polytopes]{On zero-sum polytopes: reciprocity, rigidity, and cyclic sieving}

\author{Dongchun Han}
\author{Xuan Wang}
\author{Hanbin Zhang}
\author{Shiwen Zhang}


\begin{document}

\begin{abstract}
Let $G$ be a finite abelian group of order $n$, and let $\M(G,m)$ denote the set of zero-sum sequences over $G$ of length $m$. We introduce the \emph{zero-sum polytope} $\calP_G$, a rational polytope of dimension $n-1$, whose lattice points encode zero-sum sequences:
\[
        |\M(G,m)|=|m\calP_G\cap \Z^n|.
\]
This naturally realizes the enumeration of zero-sum sequences as a problem in rational Ehrhart theory, which leads to a combinatorial reciprocity theorem identifying the negative evaluations of the corresponding counting quasipolynomial with zero-sum sequences of full support. Our main results establish a face-stratified rigidity  for zero-sum polytopes: whenever two such polytopes have equal total lattice point counts at specific dilations,
the dimension-wise open-face strata are equinumerous. Moreover, we study the natural $\Aut(G)$-action on $\calP_G$, derive equivariant generating functions and reciprocity formulas, and obtain cyclic sieving phenomena for natural cyclic actions.
\end{abstract}

\maketitle

\section{Introduction}

Let $G$ be a finite abelian group, written additively, and let $|G|=n$.
Let $S=g_1\bdot\ldots\bdot g_k$ be a {\it sequence} over $G$ (unordered, with repetitions allowed), where $g_1,\ldots,g_k\in G$ and $k$ is called the {\it length} of $S$. We define $\sigma(S)=g_1+\cdots+g_k$. We say that $S$ is a {\it zero-sum sequence} if $\sigma(S)=0_G$, the identity element of $G$. Zero-sum theory is a classical and active branch of additive number theory; see \cite{GaoGeroldingerSurvey} for a survey, and see \cite{GeroldingerMonthly,GeroldingerHalterKoch,GeroldingerGrynkiewiczZhong,GeroldingerZhong} and \cite{CziszterDomokosGeroldinger} for its interactions with factorization theory and invariant theory. For a positive integer $m$, let
\[
        \M(G,m):=\{S:\ |S|=m,\ \sigma(S)=0_G\}.
\]

The enumeration of $\M(G,m)$ has appeared in several contexts. In 1975 Fredman \cite{Fredman} proved the reciprocity
\begin{equation}\label{eq:Fredman}
        |\M(C_n,m)|=|\M(C_m,n)|,
\end{equation}
where $C_n$ denotes the cyclic group of order $n$. Elashvili, Jibladze and Pataraia \cite{ElashviliJibladze,ElashviliJibladzePataraia} later rediscovered this identity using ideas from invariant theory going back to Molien \cite{Molien} and a result of Almkvist and Fossum \cite{AlmkvistFossum}; also see \cite{Panyushev}. It was remarked in \cite[Introduction]{ElashviliJibladzePataraia} that Alon also independently proved (\ref{eq:Fredman}) using a necklace interpretation. Meanwhile, Andrews, Alon and Stanley independently obtained the counting formula for $\mathsf M(C_n,m)$; see \cite[Introduction and Section 3]{ElashviliJibladzePataraia}. In general, one has the counting formula
\begin{equation}\label{eq:MGFormula-intro}
        |\M(G,m)|=\frac{1}{n+m}\sum_{d\mid (n,m)}\varphi_G(d)
        \binom{n/d+m/d}{n/d},
\end{equation}
where $\varphi_G(d)$ denotes the number of elements of $G$ of order $d$; see \cite{HanZhang2021,MuratovicWang}. Han and Zhang \cite{HanZhang2021} gave a combinatorial proof, using rational Dyck paths, of the fact that if $G$ and $H$ are finite abelian groups with $(|G|,|H|)=1$, then
\begin{equation}\label{eq:coprime-reciprocity}
        |\M(G,|H|)|=|\M(H,|G|)|=\Cat_{|G|,|H|},
\end{equation}
where, for coprime positive integers $n,m$,
\(
        \Cat_{n,m}:=\frac{1}{n+m}\binom{n+m}{n}
\)
is the rational Catalan number. Li and Zhang \cite{LiZhang} further characterized the above reciprocity \eqref{eq:coprime-reciprocity} and related these results to the study of the group permanents.

The enumeration of zero-sum sequences is equivalent to counting nonnegative integral solutions of linear equations in finite abelian groups, which implies that the set of zero-sum sequences may naturally reveal an underlying discrete-geometric structure. A sequence $S$ over $G$ is encoded by its multiplicity vector $(a_g)_{g\in G}\in \mathbb Z_{\ge 0}^G$. The zero-sum condition is the linear equation
\(
        \sum_{g\in G} a_g g=0_G
\) in $G$. Let
\[
        \Delta_G:=\left\{x=(x_g)_{g\in G}\in \R^G_{\ge 0}:\sum_{g\in G}x_g=1\right\}
\]
be the standard simplex on the index set $G$, and define the zero-sum lattice
\[
        \Lambda_G:=\left\{(a_g)_{g\in G}\in \Z^G:\sum_{g\in G}a_g g=0\right\}.
\]
Then multiplicity vectors give a bijection
\[
        \M(G,m)\longleftrightarrow m\Delta_G\cap \Lambda_G.
\]
Since $\Lambda_G$ is a full-rank sublattice of $\Z^G$, choose a linear isomorphism $T:\R^G\to \R^n$ such that $T(\Lambda_G)=\Z^n$, and set
\[
        \calP_G:=T(\Delta_G).
\]
Then $\calP_G$ is a rational simplex of dimension $n-1$ and
\[
        |\M(G,m)|=|m\calP_G\cap \Z^n|.
\]
We call $\calP_G$ the \emph{zero-sum polytope} of $G$.
Although its face lattice is that of a simplex, the finite-index lattice induced by the zero-sum condition carries arithmetic information about $G$. Thus the geometry is combinatorially simple but arithmetically rich.

Ehrhart theory has been an important topic in discrete geometry and algebraic combinatorics, bridging the discrete world of lattice points with the continuous geometric world of convex polytopes. If $\mathcal P\subseteq \R^d$ is a lattice polytope, Ehrhart \cite{Ehrhart}  proved that the counting function
\[
L_{\mathcal{P}}(t)=|t\mathcal{P}\cap\mathbb{Z}^d|,\qquad t\in\mathbb{N},
\]
is a polynomial in \(t\) of degree \(d\). The polynomial $L_{\mathcal{P}}(t)=c_d t^d+c_{d-1}t^{d-1}+\cdots+c_0$, now called the {\it Ehrhart polynomial}, encodes fundamental geometric information: its leading coefficient is the volume of \(\mathcal{P}\), its second coefficient is half the surface area (suitably normalized), and its constant term is the Euler characteristic of \(\mathcal{P}\). The famous Ehrhart--Macdonald reciprocity \cite{Macdonald} relates the Ehrhart polynomial of a polytope to that of its interior:
\[
L_{\mathcal{P}}(-t)=(-1)^d L_{\mathcal{P}^\circ}(t),\qquad t\in\mathbb{N},
\]
where $\mathcal{P}^\circ$ denotes the interior of $\mathcal P$. This reciprocity is one of the most classical examples of  combinatorial reciprocity. We refer to \cite{BeckRobins} for background on Ehrhart theory and to \cite{Beck1,Beck2,BeckZas,BorgerKretschmerNill,ClarkeKolbl,DAliJuhnkeKoch,deVriesYoshinaga,GustafssonSolus,Hibi,MaximSchuermann,Stanley1980,Stapledon2008} for some classical results and recent developments. More generally, it is also natural and interesting to consider $|t\mathcal{P}\cap\Lambda|$ where $\Lambda$ is a full-rank sublattice of $\mathbb{Z}^d$; see, e.g., \cite{LagariasZiegler} for a classical result of Lagarias and Ziegler.
In this case, there exists a rational polytope $\mathcal P'$ (its vertices have rational coordinates) such that $|t\mathcal{P}\cap\Lambda|=|t\mathcal{P}'\cap\mathbb Z^d|$.
The study of rational polytopes and their counting functions \(L_{\mathcal{P}}(t)\) (which is no longer a polynomial but a quasipolynomial) is call rational Ehrhart theory; we refer to \cite{BeckEliaRehberg} for a recent exposition on rational Ehrhart theory.

The purpose of this paper is to develop the above Ehrhart-theoretic model for zero-sum sequences. We obtain reciprocity theorems, face-stratified rigidity results, and equivariant refinements under automorphism-group actions. Note that, although the polytope $\calP_G$ depends on the chosen basis of $\Lambda_G$, its lattice-equivalence class, Ehrhart quasipolynomial, and face-stratified lattice-point counts are independent of this choice.

\subsection{Combinatorial reciprocity}

A central phenomenon in enumerative combinatorics is \emph{combinatorial reciprocity} \cite{StanleyReciprocity}: a polynomial, or more generally a generating function, which counts combinatorial objects at positive integer arguments often acquires, after evaluation at negative integers and up to a predictable sign, an interpretation as counting a related family of objects. Classical examples include chromatic polynomials, Ehrhart polynomials, and order polynomials. See \cite{BeckSanyal} for an excellent exposition.

For $\M(G,m)$, the formula \eqref{eq:MGFormula-intro} (as a quasipolynomial, which can be evaluated at $-m$) leads to the following reciprocity theorem.

\begin{theorem}\label{thm:main-reciprocity}
Let $G$ be a finite abelian group of order $|G|=n$. For every positive integer $m$, one has
\begin{equation}\label{eq:ZS-reciprocity}
        |\M(G,-m)|=(-1)^{n-1}|\M^\circ(G,m)|,
\end{equation}
where
\[
        \M^\circ(G,m):=\{S\in \M(G,m):\supp(S)=G\}
\]
is the set of zero-sum sequences of length $m$ in which every element of $G$ occurs with positive multiplicity.
\end{theorem}

We give two proofs in Section~\ref{sec:reciprocity}. The first is a direct proof from the counting formula, and the second shows that \eqref{eq:ZS-reciprocity}
is a manifestation of the classical
Ehrhart--Macdonald reciprocity theorem, which provides a conceptual explanation for the appearance of Ehrhart theory in the enumeration of
zero-sum sequences. We also prove an equivariant analogue for invariant zero-sum sequences; see Theorem~\ref{thm:equiv-zero-sum-reciprocity}.

\subsection{Face-stratified rigidity}

Let $\mathcal F(\calP_G)$ denote the set of faces of $\calP_G$. For a face $F\in \mathcal F(\calP_G)$, write $F^\circ$ for its relative interior; see Section~2 for the precise definitions.  With this
notation, we state our first main result, which may be viewed as a rigidity
phenomenon for the face-stratified lattice-point counts of the polytopes
$\mathcal P_G$.

\begin{theorem}\label{thm:main-rigidity}
Let $G$ and $H$ be finite abelian groups with $|G|=|H|=n$, and let $m$ be a positive integer with $m\ge n$. If
\[
        |m\calP_G\cap \Z^n|=|m\calP_H\cap \Z^n|,
\]
then
\[
        |m\calP_G^\circ\cap \Z^n|=|m\calP_H^\circ\cap \Z^n|.
\]
Moreover, for every $0\le k\le n-1$, one has
\[
        \sum_{\substack{F\in \mathcal F(\calP_G)\\ \dim F=k}} |mF^\circ\cap \Z^n|
        =
        \sum_{\substack{F\in \mathcal F(\calP_H)\\ \dim F=k}} |mF^\circ\cap \Z^n|.
\]
\end{theorem}

Theorem~\ref{thm:main-rigidity} says that equality of the total lattice-point
counts of $\mathcal P_G$ and $\mathcal P_H$ at a single dilation $m$ forces
a much stronger equality: the lattice points in every relative open face
stratum, grouped by dimension, are also equinumerous. This establishes a rigidity phenomenon (rough information uniquely determines finer structure) on zero-sum polytopes. Equivalently, the
equality
$L_{\mathcal P_G}(m)=L_{\mathcal P_H}(m)$
implies $ L_{\mathcal P_G^{\circ}}(m)
=
(-1)^{n-1}L_{\mathcal P_G}(-m)
=
(-1)^{n-1}L_{\mathcal P_H}(-m)
=
L_{\mathcal P_H^{\circ}}(m).$ It is worth emphasizing that Theorem~\ref{thm:main-rigidity}  is not vacuous. For
many choices of $n$ and $m>n$, there exist non-isomorphic finite abelian
groups $G$ and $H$ of order $n$ such that
\[
|\mathsf M(G,m)|=|\mathsf M(H,m)|.
\]
For instance, when $(n,m)=1$, the formula \eqref{eq:MGFormula-intro} gives
\[
|\mathsf M(G,m)|=|\mathsf M(H,m)|=\mathsf{Cat}_{n,m}
\]
for all abelian groups $G$ and $H$ of order $n$. Further examples are
provided by \cite[Theorem~1.5]{WangZhangZhang}. 
Theorem~\ref{thm:main-rigidity} shows that such coincidences automatically extend to the finer support-stratified lattice-point structure. Moreover, in Remark \ref{finercorr}, we show that the rigidity in Theorem~\ref{thm:main-rigidity} cannot be strengthened to a face-to-face correspondence.

We also prove a reciprocal-dilation version for groups of different orders.

\begin{theorem}\label{thm:main-reciprocal-rigidity}
Let $G$ and $H$ be finite abelian groups with $|G|=n$ and $|H|=m$. Suppose that
\[
        |m\calP_G\cap \Z^n|=|n\calP_H\cap \Z^m|.
\]
Then, for every $0\le k<\min\{n,m\}$, one has
\[
        \sum_{\substack{F\in \mathcal F(\calP_G)\\ \dim F=k}} |mF^\circ\cap \Z^n|
        =
        \sum_{\substack{F\in \mathcal F(\calP_H)\\ \dim F=k}} |nF^\circ\cap \Z^m|.
\]
\end{theorem}

Theorem~\ref{thm:main-reciprocal-rigidity} says that the same rigidity principle persists even when the two zero-sum polytopes have different ambient dimensions. Equality of the total counts at the reciprocal dilations determined by the group orders forces equality of the corresponding face-stratified counts in every common dimension.

\subsection{Equivariant Ehrhart theory and cyclic sieving}

We next study zero-sum polytopes from an equivariant point of view. Equivariant Ehrhart theory, initiated by Stapledon \cite{Stapledon2011}, refines classical Ehrhart theory by incorporating symmetries of the polytope. If a finite group $\Gamma$ acts on a lattice $\Lambda$ and preserves a rational polytope $P$, then instead of considering only the scalar counting function
\(
       |rP\cap \Lambda|,
\)
one considers the permutation representation of $\Gamma$ on the finite set $rP\cap \Lambda$, or equivalently its character
\[
        \gamma\longmapsto |(rP\cap \Lambda)^\gamma|,\qquad \gamma\in\Gamma.
\]
Since Stapledon's foundational work, equivariant Ehrhart theory has developed in several directions. Explicit calculations and structural results have been obtained for families such as permutahedra \cite{ArdilaSupinaVindas}; further progress has used symmetric triangulations, zonotopal decompositions, and invariant hypersurfaces \cite{EliaKimSupina}; positivity phenomena have been established for translative group actions \cite{DAliDelucchi}; and recent work has advanced the theory for hypersimplices and invariant triangulations \cite{ClarkeKolbl,Stapledon2023}.

In this paper we apply this perspective to the natural $\Aut(G)$-action on $\calP_G$. For each automorphism of $G$, we describe the corresponding fixed-point support and derive explicit Ehrhart generating functions for invariant lattice points. We also prove an equivariant reciprocity theorem for invariant zero-sum sequences. Thus the ordinary enumeration of zero-sum sequences is refined to a symmetry-sensitive enumeration, in which lattice points are counted according to their stabilizers under the automorphism group.

A particularly important and interesting case occurs when the action is given by a cyclic group. Recall that if $X$ is a finite set equipped with an action of a cyclic group $C_N=\langle c\rangle$ of order $N$, and if $f(q)\in \Z_{\ge 0}[q]$ satisfies $f(1)=|X|$, then the triple
\[
        (X,C_N,f(q))
\]
exhibits the \emph{cyclic sieving phenomenon} \cite{RSW} if, for every integer $d$,
\[
        |X^{c^d}|=|\{x\in X:c^d\cdot x=x\}|=f(\zeta_N^d),
\]
where $\zeta_N$ is a primitive $N$th root of unity. This phenomenon generalizes Stembridge's $q=-1$ phenomenon \cite{Stem1,Stem2,Stem3} and has received a lot of attention; see \cite{ALPU,FuTangHanZeng,OhRhoades,RSW2,Sagan,ZengZhang} for some examples. In Section~\ref{sec:cyclic-sieving}, we establish several cyclic sieving results for lattice-point sets associated with $\calP_G$ under cyclic subgroups of the natural $\Aut(G)$-action. In particular, we obtain the following result; see Sections 2 and 5 for notation and definitions.

\begin{theorem}\label{thm:semiregular-CSP}
Let $G$ be a finite abelian group of order $n$, let $m\in \N$ with $(n,m)=1$, and let $c\in \Aut(G)$ have order $d$. Assume that $c$ is semiregular on $G\setminus\{0\}$. Then the triple
\[
        \bigl(m\calP_G\cap \Z^n,\langle \widetilde c\rangle,\Cat_{n,m}(q)\bigr)
\]
exhibits the cyclic sieving phenomenon.
\end{theorem}

The paper is organized as follows. Section~\ref{sec:zero-sum-polytope} introduces zero-sum polytopes and support-refined generating functions. Section~\ref{sec:reciprocity} proves the reciprocity theorems. Section~\ref{sec:rigidity} proves Theorems~\ref{thm:main-rigidity} and \ref{thm:main-reciprocal-rigidity}. Section~\ref{sec:cyclic-sieving} proves the cyclic sieving results. Section~\ref{sec:conclusion} gives concluding remarks, and Appendix~\ref{app:comparison} proves the comparison lemma used in the rigidity arguments.

\section{Zero-sum polytopes and Ehrhart theory}\label{sec:zero-sum-polytope}

Denote by $\N$ the set of positive integers and put $\Nzero:=\N\cup\{0\}$. Let $G$ be a finite abelian group of order $n$. We identify a sequence $S$ over $G$ with its multiplicity vector $(a_g)_{g\in G}\in \Nzero^G$. Under this identification,
\[
        |S|=\sum_{g\in G}a_g,\qquad
        \sigma(S)=\sum_{g\in G}a_g g,
        \qquad
        \supp(S)=\{g\in G:a_g>0\}.
\]
The simplex
\[
        \Delta_G=\left\{x=(x_g)_{g\in G}\in \R^G_{\ge 0}:\sum_{g\in G}x_g=1\right\}
\]
has dimension $n-1$. For a nonempty subset $A\subseteq G$, set
\[
        \Delta_A:=\{x\in \Delta_G:\supp(x)\subseteq A\},
        \qquad
        \Delta_A^\circ:=\{x\in \Delta_G:\supp(x)=A\}.
\]
Then $\Delta_A$ is a face of $\Delta_G$ of dimension $|A|-1$, and $\Delta_A^\circ$ is its relative interior.

Define
\[
        \Lambda_G:=\left\{(a_g)_{g\in G}\in \Z^G:\sum_{g\in G}a_g g=0\right\}.
\]
The homomorphism
\[
        \Z^G\longrightarrow G,
        \qquad
        (a_g)_{g\in G}\longmapsto \sum_{g\in G}a_g g
\]
is surjective, so $\Lambda_G$ is a full-rank sublattice of $\Z^G$ of index $|G|=n$.

\begin{proposition}\label{prop:multiplicity-bijection}
For every $m\in \N$, the multiplicity-vector map induces bijections
\[
        \M(G,m)\longleftrightarrow m\Delta_G\cap \Lambda_G,
\]
and, for every nonempty $A\subseteq G$,
\[
        \{S\in \M(G,m):\supp(S)=A\}\longleftrightarrow m\Delta_A^\circ\cap \Lambda_G.
\]
Consequently,
\[
        |\M(G,m)|=|m\Delta_G\cap \Lambda_G|,
\]
and
\[
        \sum_{S\in \M(G,m)}q^{|\supp(S)|}
        =\sum_{\emptyset\ne A\subseteq G} q^{|A|}|m\Delta_A^\circ\cap \Lambda_G|.
\]
\end{proposition}

\begin{proof}
A vector $(a_g)_{g\in G}\in \Nzero^G$ lies in $m\Delta_G$ if and only if $\sum_g a_g=m$. It lies in $\Lambda_G$ if and only if $\sum_g a_g g=0$. Hence $m\Delta_G\cap \Lambda_G$ is precisely the set of multiplicity vectors of zero-sum sequences of length $m$. The support-refined statement follows from the definition of $\Delta_A^\circ$.
\end{proof}

Choose a $\Z$-basis of $\Lambda_G$, and let
\[
        T:\R^G\longrightarrow \R^n
\]
be the linear isomorphism sending this basis to the standard basis of $\Z^n$. Thus $T(\Lambda_G)=\Z^n$. Define the \emph{zero-sum polytope} of $G$ by
\[
        \calP_G:=T(\Delta_G)\subseteq \R^n.
\]
It is a rational simplex of dimension $n-1$.

\begin{lemma}\label{lem:faces-under-linear}
Let $P\subseteq \R^d$ be a polytope and let $L:\R^d\to \R^d$ be a linear isomorphism. Then $L$ induces a dimension-preserving bijection between the faces of $P$ and the faces of $L(P)$, and for every face $F$ of $P$ one has
\[
        L(F^\circ)=(L(F))^\circ.
\]
\end{lemma}

\begin{proof}
If $F=P\cap H$ for a supporting hyperplane $H=\{x:f(x)=c\}$ of $P$, then
\[
        L(F)=L(P)\cap \{y:f(L^{-1}y)=c\},
\]
so $L(F)$ is a face of $L(P)$. Applying the same argument to $L^{-1}$ gives a bijection of face lattices. Since $L$ is a homeomorphism between affine spans, it carries relative interiors to relative interiors.
\end{proof}

\begin{proposition}\label{prop:polytope-interpretation}
Let $G$ be a finite abelian group of order $n$. For every $m\in \N$,
\[
        |\M(G,m)|=|m\calP_G\cap \Z^n|.
\]
Moreover, if $A\subseteq G$ is nonempty and $F_A:=T(\Delta_A)$, then $F_A$ is a face of $\calP_G$ of dimension $|A|-1$, and
\[
        |mF_A^\circ\cap \Z^n|
        =|m\Delta_A^\circ\cap \Lambda_G|
        =|\{S\in \M(G,m):\supp(S)=A\}|.
\]
In particular,
\[
        |m\calP_G^\circ\cap \Z^n|
        =|\{S\in \M(G,m):\supp(S)=G\}|.
\]
\end{proposition}

\begin{proof}
This follows from Proposition~\ref{prop:multiplicity-bijection}, the identity $T(\Lambda_G)=\Z^n$, and Lemma~\ref{lem:faces-under-linear}.
\end{proof}

Define the Ehrhart quasipolynomial
\[
        L_G(m):=|m\calP_G\cap \Z^n|=|\M(G,m)|.
\]

\begin{theorem}\label{thm:ehrhart-series}
Let $G$ be a finite abelian group of order $n$, and suppose
\[
        G\cong C_{n_1}\oplus\cdots\oplus C_{n_r},
        \qquad 1<n_1\mid \cdots\mid n_r.
\]
Then the Ehrhart series of $\calP_G$ is
\begin{equation}\label{eq:ehrhart-series}
        \Ehr_{\calP_G}(t):=\sum_{m\ge 0}L_G(m)t^m
        =\frac1n\sum_{d\mid n}\frac{\varphi_G(d)}{(1-t^d)^{n/d}}.
\end{equation}
Consequently, for every positive integer $m$,
\begin{equation}\label{eq:LG-formula}
        L_G(m)=|\M(G,m)|
        =\frac1n\sum_{d\mid (m,n)}\varphi_G(d)
        \binom{m/d+n/d-1}{n/d-1},
\end{equation}
where
\begin{equation}\label{eq:phi-formula}
        \varphi_G(d)=\sum_{\ell\mid d}\mu\!\left(\frac d\ell\right)
        \prod_{i=1}^r (n_i,\ell).
\end{equation}
\end{theorem}

\begin{proof}
These formulas are standard; see \cite{HanZhang2021,MuratovicWang}. Formula \eqref{eq:LG-formula} is equivalent to \eqref{eq:MGFormula-intro} by the identity
\[
        \frac1n\binom{m/d+n/d-1}{n/d-1}
        =\frac1{m+n}\binom{m/d+n/d}{n/d}.
\]
\end{proof}

\begin{remark}\label{rem:period-divides-exponent}
Formula \eqref{eq:LG-formula} immediately implies that the quasiperiod of $L_G(m)$ divides the exponent of $G$. Indeed, only divisors $d$ that occur as orders of elements of $G$ can contribute, and each such $d$ divides $\exp(G)$. Moreover, the high-degree part of $L_G(m)$ is independent of the group structure: the term $d=1$ has degree $n-1$, while each term with $d>1$ has degree at most $n/2-1$. Thus the arithmetic of $G$ is encoded in the lower-degree periodic part of the quasipolynomial. We refer to \cite{McAllisterWoods} for a related study.
\end{remark}

\subsection{Support-refined generating functions}

Define the support-refined zero-sum generating function
\[
        \calF_G(q,t):=\sum_{m\ge 0}\sum_{S\in \M(G,m)}q^{|\supp(S)|}t^m.
\]

\begin{proposition}\label{prop:FG}
Let $G$ be a finite abelian group of order $n$. Then
\begin{equation}\label{eq:FG}
        \calF_G(q,t)
        =\frac1n\sum_{d\mid n}\varphi_G(d)
        \left(\frac{1-(-1)^d(q-1)^d t^d}{1-t^d}\right)^{n/d}.
\end{equation}
\end{proposition}

\begin{proof}
Write $\widehat G$ for the character group of $G$. For any $x\in G$, character orthogonality gives
\[
        \frac1n\sum_{\chi\in \widehat G}\chi(x)
        =\begin{cases}1,&x=0,\\ 0,&x\ne 0.\end{cases}
\]
Thus the zero-sum condition may be imposed by averaging over characters. If a sequence is represented by a multiplicity vector $(a_h)_{h\in G}$, then
\[
        \mathbf 1_{\{\sum_h a_hh=0\}}
        =\frac1n\sum_{\chi\in \widehat G}\chi\!\left(\sum_{h\in G}a_hh\right).
\]
Therefore
\begin{align*}
\calF_G(q,t)
&=\sum_{(a_h)\in \Nzero^G}
  q^{|\{h:a_h>0\}|}t^{\sum_h a_h}
  \mathbf 1_{\{\sum_h a_hh=0\}}                                      \\
&=\frac1n\sum_{\chi\in \widehat G}
  \sum_{(a_h)\in \Nzero^G}
  q^{|\{h:a_h>0\}|}t^{\sum_h a_h}
  \chi\!\left(\sum_{h\in G}a_hh\right).
\end{align*}
Since $\chi$ is a homomorphism,
\[
        \chi\!\left(\sum_{h\in G}a_hh\right)=\prod_{h\in G}\chi(h)^{a_h}.
\]
The inner sum factors as
\begin{align*}
\calF_G(q,t)
&=\frac1n\sum_{\chi\in \widehat G}
  \prod_{h\in G}\left(\sum_{a\ge 0}q^{\mathbf 1_{\{a>0\}}}(t\chi(h))^a\right) \\
&=\frac1n\sum_{\chi\in \widehat G}
  \prod_{h\in G}\frac{1+(q-1)t\chi(h)}{1-t\chi(h)}.
\end{align*}
Suppose that $\chi$ has order $d$. Then its image is the group $\mu_d$ of $d$th roots of unity, and each value in $\mu_d$ is assumed exactly $n/d$ times on $G$. Hence
\[
        \prod_{h\in G}\frac{1+(q-1)t\chi(h)}{1-t\chi(h)}
        =\left(\prod_{\zeta^d=1}\frac{1+(q-1)t\zeta}{1-t\zeta}\right)^{n/d}.
\]
Using $\prod_{\zeta^d=1}(1-x\zeta)=1-x^d$, we obtain
\[
        \prod_{\zeta^d=1}(1-t\zeta)=1-t^d,
        \qquad
        \prod_{\zeta^d=1}(1+(q-1)t\zeta)=1-(-1)^d(q-1)^d t^d.
\]
Therefore the product depends only on $d$, and grouping the characters according to their orders gives \eqref{eq:FG}, because $G\cong \widehat G$ and hence the number of characters of order $d$ is $\varphi_G(d)$.
\end{proof}

\subsection{The equivariant zero-sum polytope}

We now consider zero-sum polytopes in the equivariant setting. For $a\in \Aut(G)$, let
\begin{itemize}[leftmargin=2em]
\item $\Fix_G(a)$ be the set of fixed points of $a$;
\item $\calO(a)$ be the set of orbits of the action of $a$ on $G$;
\item $r(a):=|\calO(a)|$;
\item $\ell(O):=|O|$ and $\sigma(O):=\sum_{g\in O}g\in G$ for $O\in \calO(a)$.
\end{itemize}
For $m\in \Nzero$, define
\[
        \M_a(G,m):=\{S\in \M(G,m):a\cdot S=S\},
\]
the set of zero-sum sequences of length $m$ over $G$ that are invariant under $a$.

The automorphism $a$ acts linearly on $\R^G$ by permuting coordinates:
\[
        (a\cdot x)_g:=x_{a^{-1}g}.
\]
This action preserves both $\Delta_G$ and $\Lambda_G$. Transporting the action through $T$ gives the linear automorphism
\[
        \widetilde a:=TaT^{-1}
\]
of $\R^n$, which preserves both $\Z^n$ and $\calP_G$.

\begin{proposition}\label{prop:equi-zero-sum}
Let $a\in \Aut(G)$ and let $\widetilde a=TaT^{-1}$. Then, for every $m\ge 0$,
\[
        |\M_a(G,m)|=|(m\calP_G)^{\widetilde a}\cap \Z^n|.
\]
\end{proposition}

\begin{proof}
A zero-sum sequence $S\in \M(G,m)$ corresponds to a multiplicity vector $b=(b_g)_{g\in G}\in m\Delta_G\cap \Lambda_G$. The condition $a\cdot S=S$ is exactly the condition $a\cdot b=b$. Hence
\[
        \M_a(G,m)\longleftrightarrow (m\Delta_G)^a\cap \Lambda_G.
\]
Applying $T$ gives a bijection with $(m\calP_G)^{\widetilde a}\cap \Z^n$.
\end{proof}

For later use, we describe this fixed-point set in orbit coordinates. Define
\[
        \Delta_G^a:=\left\{(x_O)_{O\in \calO(a)}\in \R_{\ge 0}^{\calO(a)}:
        \sum_{O\in \calO(a)}\ell(O)x_O=1\right\},
\]
and
\[
        \Lambda_G^a:=\left\{(c_O)_{O\in \calO(a)}\in \Z^{\calO(a)}:
        \sum_{O\in \calO(a)}c_O\sigma(O)=0\right\}.
\]
An $a$-invariant multiplicity vector is constant on each $a$-orbit. If $c_O$ is the common multiplicity on the orbit $O$, then the length and zero-sum conditions become
\[
        \sum_{O\in \calO(a)}\ell(O)c_O=m,
        \qquad
        \sum_{O\in \calO(a)}c_O\sigma(O)=0.
\]
Thus the multiplicity-on-orbits map gives a bijection
\begin{equation}\label{eq:MaGm-orbit}
        \M_a(G,m)\longleftrightarrow m\Delta_G^a\cap \Lambda_G^a.
\end{equation}

\begin{proposition}\label{prop:PGa}
The lattice $\Lambda_G^a$ is a full-rank sublattice of $\Z^{\calO(a)}$. Choose a linear isomorphism
\[
        T_a:\R^{\calO(a)}\longrightarrow \R^{r(a)}
\]
such that $T_a(\Lambda_G^a)=\Z^{r(a)}$. Then
\[
        \calP_{G,a}:=T_a(\Delta_G^a)
\]
is a rational simplex of dimension $r(a)-1$, and for every $m\ge 0$,
\[
        |\M_a(G,m)|=|m\calP_{G,a}\cap \Z^{r(a)}|.
\]
\end{proposition}

\begin{proof}
Consider the homomorphism
\[
        \psi_a:\Z^{\calO(a)}\longrightarrow G,
        \qquad
        (c_O)_O\longmapsto \sum_{O\in \calO(a)}c_O\sigma(O).
\]
Then $\Lambda_G^a=\ker(\psi_a)$. Since $G$ is finite, $\ker(\psi_a)$ has finite index in $\Z^{\calO(a)}$, so $\Lambda_G^a$ is full-rank. The set $\Delta_G^a$ is the simplex cut out by $\sum_O \ell(O)x_O=1$ in the nonnegative orthant, with vertices $e_O/\ell(O)$ for $O\in \calO(a)$; hence it has dimension $r(a)-1$. Applying $T_a$ to the bijection \eqref{eq:MaGm-orbit} proves the lattice-point formula.
\end{proof}

We now introduce the support-refined generating function for $a$-invariant zero-sum sequences:
\[
        \calF_{G,a}(q,t):=\sum_{m\ge 0}\sum_{S\in \M_a(G,m)}q^{|\supp(S)|}t^m.
\]

\begin{proposition}\label{prop:equiv-support}
Let $G$ be a finite abelian group and let $a\in \Aut(G)$. Then
\begin{equation}\label{eq:F-G-a}
        \calF_{G,a}(q,t)
        =\frac1{|G|}\sum_{\chi\in \widehat G}
        \prod_{O\in \calO(a)}
        \frac{1+(q^{\ell(O)}-1)\chi(\sigma(O))t^{\ell(O)}}
             {1-\chi(\sigma(O))t^{\ell(O)}}.
\end{equation}
In particular, the ordinary generating function
\[
        \calE_{G,a}(t):=\calF_{G,a}(1,t)=\sum_{m\ge 0}|\M_a(G,m)|t^m
\]
is
\begin{equation}\label{eq:E-G-a}
        \calE_{G,a}(t)
        =\frac1{|G|}\sum_{\chi\in \widehat G}
        \prod_{O\in \calO(a)}\frac1{1-\chi(\sigma(O))t^{\ell(O)}}.
\end{equation}
\end{proposition}

\begin{proof}
An $a$-invariant sequence is uniquely determined by assigning a nonnegative integer $c_O$ to each orbit $O\in \calO(a)$, where $c_O$ is the common multiplicity of all elements in $O$. Such a choice has length
$        \sum_{O\in \calO(a)}\ell(O)c_O,
$
sum
$
        \sum_{O\in \calO(a)}c_O\sigma(O),
$
and support size
$
        \sum_{\substack{O\in \calO(a)\\ c_O>0}}\ell(O).
$
Therefore
\[
\calF_{G,a}(q,t)
=\sum_{(c_O)\in \Nzero^{\calO(a)}}
q^{\sum_{c_O>0}\ell(O)}t^{\sum_O \ell(O)c_O}
\mathbf 1_{\{\sum_O c_O\sigma(O)=0\}}.
\]
Similar to the proof of Proposition \ref{prop:FG}, we have
\begin{align*}
\calF_{G,a}(q,t)
&=\frac1{|G|}\sum_{\chi\in \widehat G}
\sum_{(c_O)\in \Nzero^{\calO(a)}}
q^{\sum_{c_O>0}\ell(O)}t^{\sum_O \ell(O)c_O}
\prod_{O\in \calO(a)}\chi(\sigma(O))^{c_O} \\
&=\frac1{|G|}\sum_{\chi\in \widehat G}
\prod_{O\in \calO(a)}
\left(\sum_{c\ge 0}q^{\ell(O)\mathbf 1_{\{c>0\}}}
      (\chi(\sigma(O))t^{\ell(O)})^c\right).
\end{align*}
For a fixed orbit $O$, put $y_O=\chi(\sigma(O))t^{\ell(O)}$. Then
\[
        \sum_{c\ge 0}q^{\ell(O)\mathbf 1_{\{c>0\}}}y_O^c
        =1+\sum_{c\ge 1}q^{\ell(O)}y_O^c
        =\frac{1+(q^{\ell(O)}-1)y_O}{1-y_O}.
\]
Multiplying over all orbits and summing over characters proves \eqref{eq:F-G-a}. Setting $q=1$ gives \eqref{eq:E-G-a}.
\end{proof}


\section{Combinatorial reciprocity}\label{sec:reciprocity}

In this section, we prove several results on combinatorial reciprocity. We prove Theorem~\ref{thm:main-reciprocity} in two ways: first directly from the counting formula, and then conceptually from Ehrhart--Macdonald reciprocity.

\begin{proof}[Direct proof of Theorem~\ref{thm:main-reciprocity}]
Let $G$ be a finite abelian group of order $n$. By \eqref{eq:LG-formula}, the counting quasipolynomial for zero-sum sequences is
\[
        L_G(m)=|\M(G,m)|
        =\frac1n\sum_{d\mid (n,m)}\varphi_G(d)
        \binom{m/d+n/d-1}{n/d-1}.
\]
For a positive integer $m$, its value at $-m$ is
\[
        L_G(-m)
        =\frac1n\sum_{d\mid (n,m)}\varphi_G(d)
        \binom{-m/d+n/d-1}{n/d-1}.
\]
Note that
\[
        \binom{-m/d+n/d-1}{n/d-1}=(-1)^{n/d-1}\binom{m/d-1}{n/d-1},
\]
we obtain
\begin{equation}\label{eq:LG-negative-direct}
        L_G(-m)
        =\frac1n\sum_{d\mid (n,m)}\varphi_G(d)(-1)^{n/d-1}
        \binom{m/d-1}{n/d-1}.
\end{equation}

We now count zero-sum sequences of length $m$ with full support. By character orthogonality,
\[
        |\M^\circ(G,m)|
        =\frac1n\sum_{\chi\in \widehat G}[t^m]
        \prod_{h\in G}\left(\sum_{a\ge 1}(t\chi(h))^a\right).
\]
Since
\[
        \sum_{a\ge 1}(t\chi(h))^a=\frac{t\chi(h)}{1-t\chi(h)},
\]
we get
\[
        |\M^\circ(G,m)|
        =\frac1n\sum_{\chi\in \widehat G}[t^m]
        t^n\left(\prod_{h\in G}\chi(h)\right)
        \prod_{h\in G}\frac1{1-t\chi(h)}.
\]
Suppose that $\chi$ has order $d$. Then its values are the $d$th roots of unity, each occurring $n/d$ times. Hence
\[
        \prod_{h\in G}\chi(h)
        =\left(\prod_{\zeta^d=1}\zeta\right)^{n/d}
        =\bigl((-1)^{d-1}\bigr)^{n/d}=(-1)^{n-n/d},
\]
and
\[
        \prod_{h\in G}\frac1{1-t\chi(h)}=\frac1{(1-t^d)^{n/d}}.
\]
Therefore
\[
        |\M^\circ(G,m)|
        =\frac1n\sum_{d\mid n}\varphi_G(d)(-1)^{n-n/d}
        [t^{m-n}]\frac1{(1-t^d)^{n/d}}.
\]
The coefficient is zero unless $d\mid m$; for $d\mid (n,m)$ it equals
\[
        \binom{m/d-1}{n/d-1}.
\]
Thus
\begin{equation}\label{eq:full-support-direct}
        |\M^\circ(G,m)|
        =\frac1n\sum_{d\mid (n,m)}\varphi_G(d)(-1)^{n-n/d}
        \binom{m/d-1}{n/d-1}.
\end{equation}
Comparing \eqref{eq:LG-negative-direct} and \eqref{eq:full-support-direct}, and using
\[
        (-1)^{n-1}(-1)^{n-n/d}=(-1)^{n/d-1},
\]
we obtain
\[
        L_G(-m)=(-1)^{n-1}|\M^\circ(G,m)|.
\]
Since $L_G(-m)$ is denoted by $|\M(G,-m)|$, this proves the theorem.
\end{proof}

\begin{lemma}[Ehrhart--Macdonald reciprocity, \cite{Macdonald}]\label{lem:ehrhart-macdonald}
Let $\mathcal P$ be a rational polytope of dimension $d$. Then, for every positive integer $t$,
\[
        L_{\mathcal P}(-t)=(-1)^d L_{\mathcal P^\circ}(t),
\]
where $\mathcal P^\circ$ denotes the relative interior of $\mathcal P$.
\end{lemma}

\begin{proof}[Ehrhart-theoretic proof of Theorem~\ref{thm:main-reciprocity}]
By Proposition~\ref{prop:polytope-interpretation},
\[
        |\M(G,m)|=|m\calP_G\cap \Z^n|=L_{\calP_G}(m).
\]
Since $\dim \calP_G=n-1$, Lemma~\ref{lem:ehrhart-macdonald} gives
\[
        |\M(G,-m)|=L_{\calP_G}(-m)
        =(-1)^{n-1}|m\calP_G^\circ\cap \Z^n|.
\]
By Proposition~\ref{prop:polytope-interpretation}, the latter interior lattice points correspond precisely to zero-sum sequences of length $m$ with full support. Hence
\[
        |\M(G,-m)|=(-1)^{n-1}|\M^\circ(G,m)|.
\]
This completes the proof.
\end{proof}

We now record the equivariant version. Recall that
\[
        \M_a(G,m)=\{S\in \M(G,m):a\cdot S=S\}.
\]
Define
\[
        \M_a^\circ(G,m):=\{S\in \M_a(G,m):\supp(S)=G\}.
\]
Let
\[
        L_{G,a}(m):=|\M_a(G,m)|=|m\calP_{G,a}\cap \Z^{r(a)}|\qquad (m\ge 0).
\]
We write $|\M_a(G,-m)|$ for the value $L_{G,a}(-m)$ of the Ehrhart quasipolynomial of $\calP_{G,a}$.

\begin{theorem}\label{thm:equiv-zero-sum-reciprocity}
Let $G$ be a finite abelian group and let $a\in \Aut(G)$. Set $r(a)=|\calO(a)|$. Then, for every positive integer $m$,
\[
        |\M_a(G,-m)|=(-1)^{r(a)-1}|\M_a^\circ(G,m)|.
\]
Equivalently,
\[
        L_{G,a}(-m)=(-1)^{r(a)-1}|m\calP_{G,a}^\circ\cap \Z^{r(a)}|.
\]
\end{theorem}

\begin{proof}
By \eqref{eq:MaGm-orbit}, an element of $\M_a(G,m)$ is represented by a tuple $(c_O)_{O\in \calO(a)}\in \Nzero^{\calO(a)}$ satisfying
\[
        \sum_{O\in \calO(a)}\ell(O)c_O=m,
        \qquad
        \sum_{O\in \calO(a)}c_O\sigma(O)=0.
\]
After applying $T_a$, this is the same as a lattice point of $m\calP_{G,a}\cap \Z^{r(a)}$. Therefore $L_{G,a}(m)$ is the Ehrhart quasipolynomial of the rational simplex $\calP_{G,a}$, which has dimension $r(a)-1$. Ehrhart--Macdonald reciprocity gives
\[
        L_{G,a}(-m)=(-1)^{r(a)-1}|m\calP_{G,a}^\circ\cap \Z^{r(a)}|.
\]
Pulling back by $T_a$, an interior lattice point corresponds to a tuple $(c_O)_O$ with
\[
        c_O\ge 1\quad \text{for all }O\in \calO(a),
\]
and satisfying the same length and zero-sum conditions. This is precisely the condition that every element of $G$ occurs with positive multiplicity in the associated invariant zero-sum sequence. Hence
\[
        m\calP_{G,a}^\circ\cap \Z^{r(a)}\longleftrightarrow \M_a^\circ(G,m),
\]
and the theorem follows.
\end{proof}

Equivalently, define
\[
        \calE_{G,a}(t):=\sum_{m\ge 0}|\M_a(G,m)|t^m,
        \qquad
        \calE_{G,a}^\circ(t):=\sum_{m\ge 1}|\M_a^\circ(G,m)|t^m.
\]
Then Ehrhart--Macdonald reciprocity gives the rational-function identity
\begin{equation}\label{eq:equiv-ehrhart-gf-reciprocity}
        \calE_{G,a}^\circ(t)=(-1)^{r(a)}\calE_{G,a}(1/t).
\end{equation}

\begin{corollary}\label{cor:equiv-zero-sum-shift}
Assume that
\[
        s_G:=\sum_{h\in G}h=0.
\]
Then, for every integer $m\ge |G|$,
\[
        |\M_a(G,-m)|=(-1)^{r(a)-1}|\M_a(G,m-|G|)|.
\]
Equivalently,
\[
        \calE_{G,a}^\circ(t)=t^{|G|}\calE_{G,a}(t)=(-1)^{r(a)}\calE_{G,a}(1/t).
\]
\end{corollary}

\begin{proof}
Under the assumption $s_G=0$, the map
\[
        (c_O)_{O\in \calO(a)}\longmapsto (c_O-1)_{O\in \calO(a)}
\]
is a bijection from the set of interior orbit multiplicity vectors of length $m$ to the set of nonnegative orbit multiplicity vectors of length $m-|G|$ satisfying the zero-sum condition. Indeed,
\[
        \sum_{O\in \calO(a)}\ell(O)=|G|,
        \qquad
        \sum_{O\in \calO(a)}\sigma(O)=s_G=0.
\]
Therefore $\M_a^\circ(G,m)$ is in bijection with $\M_a(G,m-|G|)$, and the result follows from Theorem~\ref{thm:equiv-zero-sum-reciprocity}.
\end{proof}

\section{Rigidity}\label{sec:rigidity}

In this section we prove the face-stratified rigidity results.

For $m\in \Nzero$, define
\[
        S_{G,m}(q):=\sum_{S\in \M(G,m)}q^{|\supp(S)|}.
\]

\begin{proposition}\label{prop:support-polynomial-formula}
Let $G$ be a finite abelian group of order $n$. Then
\begin{equation}\label{eq:support-polynomial-formula}
S_{G,m}(q)
=\frac1n\sum_{d\mid (n,m)}\varphi_G(d)
\sum_{j=0}^{\min\{n/d,m/d\}}
\binom{n/d}{j}(-1)^{j(d+1)}(q-1)^{dj}
\binom{m/d-j+n/d-1}{n/d-1}.
\end{equation}
\end{proposition}

\begin{proof}
The polynomial $S_{G,m}(q)$ is the coefficient of $t^m$ in $\calF_G(q,t)$. Fix $d\mid n$. The $d$th summand in \eqref{eq:FG} is
\[
        \left(\frac{1-(-1)^d(q-1)^dt^d}{1-t^d}\right)^{n/d}.
\]
Expanding the numerator and denominator gives
\[
        (1-(-1)^d(q-1)^dt^d)^{n/d}
        =\sum_{j=0}^{n/d}\binom{n/d}{j}(-1)^{j(d+1)}(q-1)^{dj}t^{dj},
\]
and
\[
        (1-t^d)^{-n/d}=\sum_{r\ge 0}\binom{n/d+r-1}{r}t^{dr}.
\]
A nonzero contribution to $[t^m]$ occurs only when $d\mid m$, and coefficient extraction gives \eqref{eq:support-polynomial-formula}.
\end{proof}


\begin{proposition}\label{prop:support-geometric}
For every $m\in \N$,
\[
        S_{G,m}(q)=\sum_{\emptyset\ne F\in \mathcal F(\calP_G)}q^{\dim F+1}|mF^\circ\cap \Z^n|.
\]
\end{proposition}

\begin{proof}
By Proposition~\ref{prop:polytope-interpretation}, the coefficient of $q^{|A|}$ in $S_{G,m}(q)$ is $|mF_A^\circ\cap \Z^n|$, where $F_A=T(\Delta_A)$ is a face of dimension $|A|-1$. Summing over nonempty subsets $A\subseteq G$ proves the formula.
\end{proof}

Proposition~\ref{prop:support-geometric} explains the mechanism behind rigidity. The face lattice of $\calP_G$ is always that of a simplex; the arithmetic distribution of lattice points across the relative interiors of the faces is what varies with $G$.

\begin{theorem}\label{thm:Li-Zhang}
Let $G$ and $H$ be finite abelian groups. Then
\[
        |\M(G,|H|)|=|\M(H,|G|)|
\]
if and only if
\[
        \varphi_G(d)=\varphi_H(d)\qquad \text{for all }d\mid (|G|,|H|).
\]
\end{theorem}

\begin{proof}
This is \cite[Theorem~3]{LiZhang}.
\end{proof}

The following comparison lemma is the main arithmetic input for Theorem~\ref{thm:main-rigidity}. Its proof follows the first-difference strategy of \cite{LiZhang}; we include the details in Appendix~\ref{app:comparison}.

\begin{lemma}\label{lem:rigidity-order-distribution}
Let $G$ and $H$ be finite abelian groups with $|G|=|H|=n$, and let $m\ge n$. Then
\[
        |\M(G,m)|=|\M(H,m)|
\]
if and only if
\[
        \varphi_G(d)=\varphi_H(d)\qquad \text{for every }d\mid (n,m).
\]
\end{lemma}

\begin{proof}[Proof of Theorem~\ref{thm:main-rigidity}]
Assume that $|m\calP_G\cap \Z^n|=|m\calP_H\cap \Z^n|$. Equivalently, $|\M(G,m)|=|\M(H,m)|$. By Lemma~\ref{lem:rigidity-order-distribution},
\[
        \varphi_G(d)=\varphi_H(d)
        \qquad \text{for every }d\mid (n,m).
\]
Proposition~\ref{prop:support-polynomial-formula} then gives
\[
        S_{G,m}(q)=S_{H,m}(q).
\]
Comparing the coefficients of $q^s$ for $1\le s\le n$ shows that
\[
        |\{S\in \M(G,m):|\supp(S)|=s\}|
        =|\{T\in \M(H,m):|\supp(T)|=s\}|.
\]
By Proposition~\ref{prop:support-geometric}, support size $s$ corresponds to relative open faces of dimension $s-1$. Hence, for every $0\le k\le n-1$,
\[
        \sum_{\substack{F\in \mathcal F(\calP_G)\\ \dim F=k}} |mF^\circ\cap \Z^n|
        =
        \sum_{\substack{F\in \mathcal F(\calP_H)\\ \dim F=k}} |mF^\circ\cap \Z^n|.
\]
Taking $k=n-1$ gives the equality of interior lattice-point counts.
\end{proof}

\begin{remark}\label{finercorr}
Let $m=9$ and $n=8$, consider the two groups $G=C_2 \oplus C_4$ and $H=C_2 \oplus C_2 \oplus C_2$.
For a nonempty subset $A_1 \subseteq G$ (resp. $A_2 \subseteq H$ ), let $A_1^c$ (resp. $A_2^c$) denote the complement of $A_1$ (resp. $A_2$) in $G$ (resp. $H$).
Consider all faces of dimension $6$ (so a lattice point in such a relative open face has exactly one zero coordinate, and we use the complements (a singleton set in this case) to index the subsets for convenience), we obtain the following counts:
\begin{table}[ht]
\centering
\resizebox{\textwidth}{!}{
$\begin{array}{|c|c|c|c|c|c|c|c|c|}
\hline
A_1^c
& (0,0) & (0,1) & (0,2) & (0,3)
& (1,0) & (1,1) & (1,2) & (1,3) \\
\hline
| mF_{A_1}^\circ \cap \mathbb{Z}^n |
& 5 & 3 & 5 & 3 & 3 & 3 & 3 & 3 \\
\hline
A_2^c
& (0,0,0) & (0,0,1) & (0,1,0) & (0,1,1)
& (1,0,0) & (1,0,1) & (1,1,0) & (1,1,1) \\
\hline
| mF_{A_2}^\circ \cap \mathbb{Z}^n |
& 7 & 3 & 3 & 3 & 3 & 3 & 3 & 3 \\
\hline
\end{array}$
}
\end{table}

The sums of the numbers of lattice points in the relative interiors of all $6$-dimensional faces of the two polytopes are both equal to 28. However, they are different as multisets. It follows that the rigidity in Theorem~\ref{thm:main-rigidity} can not be strengthened to a face to face correspondence directly.
\end{remark}

\begin{proposition}\label{prop:reciprocal-support}
Let $G$ and $H$ be finite abelian groups. Then
\begin{equation}\label{eq:reciprocal-support}
        \sum_{S\in \M(G,|H|)}q^{|\supp(S)|}
        =
        \sum_{T\in \M(H,|G|)}q^{|\supp(T)|}
\end{equation}
if and only if
\[
        |\M(G,|H|)|=|\M(H,|G|)|.
\]
\end{proposition}

\begin{proof}
The implication from \eqref{eq:reciprocal-support} to equality of total counts follows by setting $q=1$. Conversely, assume that the total counts are equal. By Theorem~\ref{thm:Li-Zhang},
\[
        \varphi_G(d)=\varphi_H(d)
        \qquad \text{for every }d\mid (|G|,|H|).
\]
Let $A=|G|/d$ and $B=|H|/d$. The contribution of the divisor $d$ to $[t^{|H|}]\calF_G(q,t)$ is
\[
        \frac{\varphi_G(d)}{|G|}\,C_d(A,B),
\]
where
\[
        C_d(A,B):=[u^B]\left(\frac{1-(-1)^d(q-1)^d u}{1-u}\right)^A.
\]
Similarly, the contribution of $d$ to $[t^{|G|}]\calF_H(q,t)$ is
\[
        \frac{\varphi_H(d)}{|H|}\,C_d(B,A).
\]
Thus it is enough to prove
\begin{equation}\label{eq:Cd-symmetry}
        \frac1{|G|}C_d(A,B)=\frac1{|H|}C_d(B,A),
\end{equation}
or equivalently $B C_d(A,B)=A C_d(B,A)$.

Let $\alpha_d=(-1)^d(q-1)^d$. Then
\[
        C_d(A,B)=[u^B]\left(\frac{1-\alpha_d u}{1-u}\right)^A.
\]
Expanding gives
\[
        C_d(A,B)=\sum_{j=0}^{\min\{A,B\}}\binom{A}{j}(-\alpha_d)^j
        \binom{A+B-j-1}{B-j}.
\]
Likewise,
\[
        C_d(B,A)=\sum_{j=0}^{\min\{A,B\}}\binom{B}{j}(-\alpha_d)^j
        \binom{A+B-j-1}{A-j}.
\]
For every $0\le j\le \min\{A,B\}$,
\[
        B\binom{A}{j}\binom{A+B-j-1}{B-j}
        =A\binom{B}{j}\binom{A+B-j-1}{A-j},
\]
which is an immediate factorial computation. Therefore $B C_d(A,B)=A C_d(B,A)$, proving \eqref{eq:Cd-symmetry}. Summing over $d\mid (|G|,|H|)$ gives \eqref{eq:reciprocal-support}.
\end{proof}

\begin{proof}[Proof of Theorem~\ref{thm:main-reciprocal-rigidity}]
The hypothesis is equivalent to
\[
        |\M(G,|H|)|=|\M(H,|G|)|.
\]
By Proposition~\ref{prop:reciprocal-support},
\[
        S_{G,|H|}(q)=S_{H,|G|}(q).
\]
Comparing coefficients of $q^s$ for $1\le s\le \min\{n,m\}$ and applying Proposition~\ref{prop:support-geometric} gives, for every $0\le k<\min\{n,m\}$,
\[
        \sum_{\substack{F\in \mathcal F(\calP_G)\\ \dim F=k}} |mF^\circ\cap \Z^n|
        =
        \sum_{\substack{F\in \mathcal F(\calP_H)\\ \dim F=k}} |nF^\circ\cap \Z^m|.
\]
\end{proof}

\section{Cyclic sieving phenomenon}\label{sec:cyclic-sieving}

In this section we relate the equivariant Ehrhart theory of zero-sum polytopes to the cyclic sieving phenomenon.

\begin{definition}[Reiner-Stanton-White, \cite{RSW}]\label{def:CSP}
Let $X$ be a finite set with an action of a cyclic group $C_N=\langle c\rangle$ of order $N$. Let $f(q)\in \Z_{\ge 0}[q]$ satisfy $f(1)=|X|$, and let $\zeta_N$ be a primitive $N$th root of unity. The triple
\[
        (X,C_N,f(q))
\]
exhibits the \emph{cyclic sieving phenomenon} if, for every integer $j$,
\[
        |\{x\in X:c^j\cdot x=x\}|=f(\zeta_N^j).
\]
\end{definition}

For a positive integer $n$, define
\[
        [n]_q=1+q+\cdots+q^{n-1},
        \qquad
        [n]!_q=[n]_q[n-1]_q\cdots [1]_q,
\]
and
\[
        \qbinom{n}{k}=\frac{[n]!_q}{[k]!_q[n-k]!_q}.
\]
For coprime positive integers $n,m$, define the rational $q$-Catalan polynomial
\[
        \Cat_{n,m}(q)=\frac1{[n+m]_q}\qbinom{n+m}{m}.
\]
It is known that $\Cat_{n,m}(q)$ is a polynomial in $q$ with nonnegative integer coefficients; see \cite[Proposition~2.5.2]{Haiman1994}. In algebraic combinatorics, the study of $q$-analogues (or $q,t$-analogues) of (rational) Catalan numbers has received a lot of attention; see \cite{Bergeron,GarsiaHaiman,Haglund,Haiman2001,Haiman2002}.

\begin{proposition}\label{prop:fixed-point-free}
Let $a\in \Aut(G)$, and assume that $\Fix_G(a)=\{0\}$. Then $\sigma(O)=0$ for every nonzero orbit $O\in \calO(a)$, and
\[
        \calE_{G,a}(t)=\prod_{O\in \calO(a)}\frac1{1-t^{\ell(O)}}.
\]
If all nonzero orbits in $\calO(a)$ have common size $d$, then
\begin{equation}\label{eq:fixed-point-count-common-size}
        |\M_a(G,m)|=\binom{(n-1)/d+\lfloor m/d\rfloor}{(n-1)/d}.
\end{equation}
\end{proposition}

\begin{proof}
If $O$ is a nonzero orbit, then
\[
        a(\sigma(O))=\sigma(O).
\]
Since $\Fix_G(a)=\{0\}$, it follows that $\sigma(O)=0$. The orbit $\{0\}$ contributes the factor $(1-t)^{-1}$ in \eqref{eq:E-G-a}, and every nonzero orbit $O$ contributes $(1-t^{\ell(O)})^{-1}$. This proves the product formula.

If all nonzero orbits have size $d$, then there are $(n-1)/d$ of them. Hence
\[
        \calE_{G,a}(t)=\frac1{(1-t)(1-t^d)^r},
        \qquad r=\frac{n-1}{d}.
\]
Coefficient extraction gives
\[
        [t^m]\frac1{(1-t)(1-t^d)^r}
        =\sum_{j=0}^{\lfloor m/d\rfloor}\binom{r+j-1}{r-1}
        =\binom{r+\lfloor m/d\rfloor}{r},
\]
which is \eqref{eq:fixed-point-count-common-size}.
\end{proof}

For $f(q)=\qbinom{m}{n},\ g(q)=[n]_q\in \Z[q]$, we denote
$\genfrac{[}{]}{0pt}{}{m}{n}_{\zeta}\coloneqq f(\zeta)$ and $[n]_\zeta\coloneqq g(\zeta)$ for simplicity, where $\zeta\in\C$.

\begin{lemma}\label{lem:q-lucas-catalan}
Let $n,m\in \N$ with $(n,m)=1$, let $k\mid (n-1)$, and let $\zeta$ be a primitive $k$th root of unity. Then
\[
        \Cat_{n,m}(\zeta)=\binom{(n-1)/k+\lfloor m/k\rfloor}{(n-1)/k}.
\]
\end{lemma}

\begin{proof}
Note that $ \Cat_{n,m}(q)=\frac1{[n]_q}\qbinom{n+m-1}{n-1}$.
Write $n=kr+1$ and $m=ks+b$ with $0\le b<k$. Since $[n]_{\zeta}=1$, one has
\[
        \Cat_{n,m}(\zeta)=\genfrac{[}{]}{0pt}{}{n+m-1}{n-1}_{\zeta}.
\]
Now $n+m-1=k(r+s)+b$ and $n-1=kr$. Following the same approach as \cite[Section~2]{Sagan}, we have
\[
        \genfrac{[}{]}{0pt}{}{n+m-1}{n-1}_{\zeta}
        =\binom{r+s}{r}
        \genfrac{[}{]}{0pt}{}{b}{0}_{\zeta}
        =\binom{r+s}{r}.
\]
Since $s=\lfloor m/k\rfloor$, the result follows.
\end{proof}

\begin{lemma}\label{lem:q-minus-one-catalan}
Let $n=2r$ be even, let $m=2s+1$ be odd, and let $\zeta=-1$. Then
\[
        \Cat_{n,m}(-1)=\binom{r+s}{r}.
\]
\end{lemma}

\begin{proof}
Since $m$ is odd and $[m]_{-1}=1$,
\[
        \Cat_{n,m}(-1)=
        \genfrac{[}{]}{0pt}{}{2(r+s)}{2r}_{-1}.
\]
It is easy to see that
\[
        \genfrac{[}{]}{0pt}{}{2(r+s)}{2r}_{-1}=\binom{r+s}{r}.
\]
\end{proof}

\begin{definition}\label{def:semiregular}
Let $c\in \Aut(G)$ have finite order $d$. We say that $c$ is \emph{semiregular on $G\setminus\{0\}$} if
\[
        \Fix_G(c^j)=\{0\}
        \qquad \text{for every }1\le j<d.
\]
Equivalently, every nonzero orbit of the cyclic group $\langle c\rangle$ has size $d$.
\end{definition}

\begin{proof}[Proof of Theorem \ref{thm:semiregular-CSP}]
If $d\mid j$, then $\widetilde c^{\,j}$ is the identity, and the required equality is
\[
        |m\calP_G\cap \Z^n|=\Cat_{n,m}(1),
\]
which follows from \eqref{eq:coprime-reciprocity}. Assume now that $d\nmid j$, and put
\[
        r:=\ord(c^j)=\frac{d}{(d,j)}.
\]
Since $c$ is semiregular, so is $c^j$ on $G\setminus\{0\}$. Hence every nonzero orbit of $c^j$ has size $r$, and Proposition~\ref{prop:fixed-point-free} gives
\[
        |(m\calP_G)^{\widetilde c^{\,j}}\cap \Z^n|
        =\binom{(n-1)/r+\lfloor m/r\rfloor}{(n-1)/r}.
\]
Since $r\mid d$ and $d\mid (n-1)$, Lemma~\ref{lem:q-lucas-catalan} applies. The root $\zeta_d^j$ has order $r$, so
\[
        \Cat_{n,m}(\zeta_d^j)
        =\binom{(n-1)/r+\lfloor m/r\rfloor}{(n-1)/r}.
\]
Thus the fixed-point count agrees with the required root-of-unity evaluation for every group element.
\end{proof}

The following result shows that the semiregular hypothesis in Theorem \ref{thm:semiregular-CSP} is essentially sharp.

\begin{theorem}\label{thm:semiregular-converse}
Let $a\in \Aut(G)$ have order $d$, and assume that $d\mid (n-1)$, where $n=|G|$. Suppose that
\[
        |\M_a(G,k)|=\Cat_{n,k}(\zeta_d)
\]
holds for infinitely many positive integers $k$ with $(n,k)=1$, where $\zeta_d$ is a primitive $d$th root of unity. Then $a$ is semiregular on $G\setminus\{0\}$.
\end{theorem}

\begin{proof}
By Proposition~\ref{prop:PGa}, $|\M_a(G,k)|$ is the Ehrhart quasipolynomial of the rational simplex $\calP_{G,a}$ of dimension $r(a)-1$. On any fixed residue class modulo $d$, Lemma~\ref{lem:q-lucas-catalan} shows that $k\mapsto \Cat_{n,k}(\zeta_d)$ is a polynomial of degree $(n-1)/d$. Note that, on every residue class on which $m\calP_{G,a}\cap\Z^{r(a)}$
 is nonempty for infinitely many $m$, the corresponding Ehrhart quasipolynomial has degree $\dim \calP_{G,a}=r(a)-1$ and positive leading coefficient.
Since the equality holds for infinitely many coprime $k$, there is a residue class modulo a common period of the Ehrhart quasipolynomial and $d$ containing infinitely many such $k$. On this residue class, both sides are polynomials and agree at infinitely many integers; hence their degrees agree:
\[
        r(a)-1=\frac{n-1}{d}.
\]
Every nonzero orbit of $a$ has size dividing $d$, hence at most $d$. Since $0$ is fixed,
\[
        n=1+\sum_{O\ne \{0\}}|O|
        \le 1+d(r(a)-1)
        =1+d\cdot \frac{n-1}{d}=n.
\]
Equality throughout forces every nonzero orbit to have size exactly $d$, so $a$ is semiregular on $G\setminus\{0\}$.
\end{proof}

\begin{theorem}\label{thm:odd-inversion}
Let $G$ be a finite abelian group of odd order $n$, and let $m\in \N$ with $(n,m)=1$. Let $\tau\in \Aut(G)$ be inversion, $\tau(g)=-g$. Then the triple
\[
        \bigl(m\calP_G\cap \Z^n,\langle \widetilde\tau\rangle,\Cat_{n,m}(q)\bigr)
\]
exhibits the cyclic sieving phenomenon.
\end{theorem}

\begin{proof}
If $|G|$ is odd, inversion has no nonzero fixed point and has order $2$. Thus $\tau$ is semiregular on $G\setminus\{0\}$, and the claim follows from Theorem~\ref{thm:semiregular-CSP}.
\end{proof}

The inversion involution can be described explicitly for every finite abelian group. Let
\[
        G[2]:=\{g\in G:2g=0\}
\]
be the $2$-torsion subgroup of $G$.

\begin{proposition}\label{prop:inversion}
Let $G$ be a finite abelian group and $\tau\in\Aut(G)$ with $\tau(g)=-g$ for every $g\in G$. Then
\begin{equation}\label{eq:inversion-F}
        \calF_{G,\tau}(q,t)
        =\calF_{G[2]}(q,t)\left(\frac{1+(q^2-1)t^2}{1-t^2}\right)^{(n-|G[2]|)/2}.
\end{equation}
In particular,
\begin{equation}\label{eq:inversion-E}
        \calE_{G,\tau}(t)=\Ehr_{\calP_{G[2]}}(t)(1-t^2)^{-(n-|G[2]|)/2}.
\end{equation}
\end{proposition}

\begin{proof}
The $\tau$-fixed elements are exactly the elements of $G[2]$. Every element outside $G[2]$ belongs to a two-element orbit $\{g,-g\}$. Hence a $\tau$-invariant sequence decomposes uniquely as
\[
        S=S_{G[2]}\cdot \prod_i (g_i(-g_i))^{x_i},
        \qquad x_i\in \Nzero,
\]
where the product runs over representatives of the two-element orbits in $G\setminus G[2]$. Each such pair contributes zero to the sum, length $2$, and, if selected, support size $2$. Thus $S$ is zero-sum if and only if $S_{G[2]}$ is a zero-sum sequence over $G[2]$. Summing independently over the pair multiplicities gives the factor in \eqref{eq:inversion-F}; setting $q=1$ gives \eqref{eq:inversion-E}.
\end{proof}

\begin{corollary}\label{cor:cyclic-inversion}
Let $C_n$ be a cyclic group of order $n$, let $m\in \N$ with $(n,m)=1$, and let $\tau\in \Aut(C_n)$ be the inversion automorphism $\tau(g)=-g$. Then the triple
\[
        \bigl(m\calP_{C_n}\cap \Z^n,\langle \widetilde\tau\rangle,\Cat_{n,m}(q)\bigr)
\]
exhibits the cyclic sieving phenomenon.
\end{corollary}

\begin{proof}
If $n$ is odd, this is Theorem~\ref{thm:odd-inversion}. Assume that $n$ is even. Then $C_n[2]\cong C_2$, and
\[
        \Ehr_{\calP_{C_2}}(t)=\sum_{\ell\ge 0}|\M(C_2,\ell)|t^\ell
        =\frac1{(1-t)(1-t^2)}.
\]
By Proposition~\ref{prop:inversion},
\[
        \calE_{C_n,\tau}(t)=\frac1{(1-t)(1-t^2)^{n/2}}.
\]
Since $(n,m)=1$, the integer $m$ is odd. Write $n=2r$ and $m=2s+1$. Then
\[
        |\M_\tau(C_n,m)|
        =[t^m]\frac1{(1-t)(1-t^2)^r}
        =\sum_{j=0}^s\binom{r+j-1}{r-1}
        =\binom{r+s}{r}.
\]
By Lemma~\ref{lem:q-minus-one-catalan}, this equals $\Cat_{n,m}(-1)$. This verifies the nontrivial fixed-point condition for the order-two action, and the identity element condition is $\Cat_{n,m}(1)=|m\calP_{C_n}\cap \Z^n|$.
\end{proof}

\begin{corollary}\label{cor:prime-CSP}
Let $p$ be an odd prime, and let $m\in \N$ with $(p,m)=1$. Then the triple
\[
        \bigl(m\calP_{C_p}\cap \Z^p,\Aut(C_p),\Cat_{p,m}(q)\bigr)
\]
exhibits the cyclic sieving phenomenon.
\end{corollary}

\begin{proof}
The group $\Aut(C_p)\cong C_{p-1}$ is cyclic. Choose a generator $c$. Since $p$ is prime, every nontrivial power of $c$ acts without nonzero fixed points on $C_p$. Thus $c$ is semiregular on $C_p\setminus\{0\}$, and the claim follows from Theorem~\ref{thm:semiregular-CSP}.
\end{proof}

\begin{corollary}\label{cor:vector-space-CSP}
Let $G=(\Z/p\Z)^r$ with $p$ prime, let $\lambda\in (\Z/p\Z)^\times$ have multiplicative order $d$, and let $c$ be the automorphism $x\mapsto \lambda x$ of $G$. If $(p,m)=1$, then
\[
        \bigl(m\calP_G\cap \Z^{p^r},\langle \widetilde c\rangle,\Cat_{p^r,m}(q)\bigr)
\]
exhibits the cyclic sieving phenomenon.
\end{corollary}

\begin{proof}
For $1\le j<d$, the scalar $\lambda^j-1$ is nonzero in $\mathbb F_p$, hence multiplication by $\lambda^j-1$ is bijective on the vector space $G$. Therefore $\Fix_G(c^j)=\{0\}$, so $c$ is semiregular on $G\setminus\{0\}$. The claim follows from Theorem~\ref{thm:semiregular-CSP}.
\end{proof}

\section{Concluding remarks}\label{sec:conclusion}

This paper places the enumeration of zero-sum sequences over finite abelian groups into the framework of rational Ehrhart theory. The zero-sum polytope $\calP_G$ provides a geometric model for the counting function $|\M(G,m)|$, while the faces of $\calP_G$ record the support structure of zero-sum sequences. This viewpoint explains the reciprocity formula for $|\M(G,-m)|$ as an instance of Ehrhart--Macdonald reciprocity and leads to face-stratified rigidity phenomena for lattice-point counts.  Several further questions arise naturally.

First, it would be very interesting to generalize the geometric model $\mathcal P_G$ to the non-abelian setting. There are two possible routes. One is in additive combinatorics, the natural analogue of zero-sum sequence in the non-abelian setting is the product-one sequence (see \cite{GeroldingerOh,HanZhang2019}). The problem is that the study of product-one sequences does not seem to have been developed in a comparable enumerative form. The second route is through polynomial invariant theory, the study of polynomial invariants (in particular, on the Noether number) is closely related to zero-sum sequences (see \cite{CziszterDomokosGeroldinger,HanZhang2019}). Moreover, \cite[Section 4]{LiZhang} showed that the counting formula of certain polynomial invariants for general finite group has the same form as \eqref{eq:MGFormula-intro}.


Second, it is important and appealing to consider combinatorial interpretations for Theorems \ref{thm:main-rigidity} and \ref{thm:main-reciprocal-rigidity} . More precisely, constructing an explicit correspondence between $\mathsf M(G,m)$ (resp. $\mathsf M(G,|H|)$) and $\mathsf M(H,m)$ (resp. $\mathsf M(H,|G|)$) which preserve the support statistic, in the non-coprime setting.

Third, in Section 4, the statistic $\operatorname{supp}$ is crucial in our proof of the rigidity theorems. In fact, the statistic $\operatorname{supp}$ has been an important and widely used statistic in additive combinatorics. It would be interesting to see if there are some other statistics in additive combinatorics that satisfy analogues of Proposition \ref{prop:reciprocal-support}.


Finally, the equivariant theory developed here suggests a broader representation-theoretic refinement. The natural $\Aut(G)$-action on $\calP_G$ gives permutation representations on the sets $m\calP_G\cap \Z^n$. It would be interesting to study the corresponding equivariant $h^*$-series in the sense of Stapledon \cite{Stapledon2011} and to determine whether positivity, unimodality, or representation-stability phenomena occur for special families of groups.

\appendix

\section{Proof of Lemma \ref{lem:rigidity-order-distribution}}\label{app:comparison}

In this appendix, we provide a proof of Lemma \ref{lem:rigidity-order-distribution}.  The proof strategy is similar to (and simpler than) that of \cite[Theorem 3]{LiZhang}; we provide all the details below for the convenience of the readers.

Let $G$ be a finite abelian group. Assume that $G\cong C_{n_1}\oplus\cdots\oplus C_{n_r}$, where  $1<n_1|\cdots|n_r\in\mathbb{N}$.
Moreover, this sequence $(n_1,\dots,n_r)$ is uniquely determined by $G$. For any element $g \in G$, we write $\mathrm{ord}(g)$ for its order. Given a prime $p$, we denote by $\mathrm{Syl}_p(G)$ the Sylow $p$-subgroup of $G$.

We provide some auxiliary lemmas, which will be repeatedly used in the subsequent proofs.

\begin{lemma}{\rm(\cite{LiZhang} Lemma 5)}\label{lemma:Binom_ineq}
Let $m,n,a,b\geq 2$ be integers and $a,b|(m,n).$
\begin{enumerate}
   \item [{\rm(i)}] If $b>a$,  we have
$$
  \binom{\frac{m+n}{a}}{\frac{m}{a},\frac{n}{a}}\Big/ \binom{\frac{m+n}{b}}{\frac{m}{b},\frac{n}{b}}\geq (1+\frac{m}{n})^{n(\frac{1}{a}-\frac{1}{b})}(1+\frac{a}{b}\frac{n}{m})^{m(\frac{1}{a}-\frac{1}{b})}.
$$
Consequently, $\binom{\frac{m+n}{a}}{\frac{m}{a},\frac{n}{a}}> \binom{\frac{m+n}{b}}{\frac{m}{b},\frac{n}{b}}.$

    \item  [{\rm(ii)}] If $b\geq 2a$,  we   have $$
 a \binom{\frac{m+n}{a}}{\frac{m}{a},\frac{n}{a}}> \max\{m,n\}\binom{\frac{m+n}{b}}{\frac{m}{b},\frac{n}{b}}.
$$

\end{enumerate}
\end{lemma}

\begin{lemma}\label{lemma:Impartant_estimation}
Let $n=p^\alpha q^\beta n', m=p^\gamma q^\delta m'$ be positive integers, where $p$ and $q$ are distinct primes, $(n',pq)=(m',pq)=1$, and $\alpha, \beta, \gamma, \delta\in\mathbb N$. Suppose that $a=p^s$, $b=q^t$ ($s,t\geq 1$) satisfy $b<2a$ and  $a,b\mid (n,m)$. Denote
$$\Delta_{n,m}(a,b):=p^{\alpha+\gamma-2s-1}q^{\delta-t}n'm'.$$
\begin{enumerate}
   \item [{\rm(i)}]  If
  $n(\frac{1}{a}-\frac{1}{b})\geq 3$ and    $\{a,b\}\neq \{2,3\},$
  then
$$
     a\binom{\frac{n+m}{a}}{\frac{n}{a},\frac{m}{a}}
     \Big/\binom{\frac{n+m}{b}}{\frac{n}{b},\frac{m}{b}}>  2\Delta_{n,m}(a,b) q^\beta.
$$
Consequently, if $\Delta_{n,m}(a,b)\geq 1,$ then we have
\begin{equation}\label{eq:Bino_a_2b}
     a\binom{\frac{n+m}{a}}{\frac{n}{a},\frac{m}{a}}- (q^\beta-q^t) \binom{\frac{n+m}{b}}{\frac{n}{b},\frac{m}{b}}> 2b\binom{\frac{n+m}{b}}{\frac{n}{b},\frac{m}{b}}.
     \end{equation}
   Moreover, \eqref{eq:Bino_a_2b} always holds when $\{a,b\}= \{2,3\}$.
 \item [{\rm(ii)}]
  If
  $n(\frac{1}{a}-\frac{1}{b})=2,$ and   $\{a,b\}\neq \{2,3\},$
  then
    $$
 a\binom{\frac{n+m}{a}}{\frac{n}{a},\frac{m}{a}}
 \Big/\binom{\frac{n+m}{b}}{\frac{n}{b},\frac{m}{b}}>     \Delta_{n,m}(a,b) q^\beta .
$$
Consequently, if $\Delta_{n,m}(a,b)\geq 1,$ then we have
\begin{equation}\label{eq:Bino_a_b}
     a\binom{\frac{n+m}{a}}{\frac{n}{a},\frac{m}{a}}- (q^\beta-q^t) \binom{\frac{n+m}{b}}{\frac{n}{b},\frac{m}{b}}> b\binom{\frac{n+m}{b}}{\frac{n}{b},\frac{m}{b}}.
     \end{equation}
  \end{enumerate}

\end{lemma}
\begin{proof}
Without loss of generality, we assume that $n\geq m$.

(i). Let
\[
        \eta=\frac1a-\frac1b,\qquad X=n\eta,\qquad Y=m\eta .
\]
Since \(a,b\mid (n,m)\), both \(X\) and \(Y\) are positive integers. Moreover,
as \(n\ge m\), we have \(X\ge Y\). By hypothesis,
\(X=n(1/a-1/b)\ge 3\).

By Lemma \ref{lemma:Binom_ineq}.(i), we have
\[
a\binom{(n+m)/a}{n/a,m/a}\Big/
        {\binom{(n+m)/b}{n/b,m/b}}
\ge
a\left(1+\frac mn\right)^X
 \left(1+\frac ab\frac nm\right)^Y .
\]
As \(m/n=Y/X\), it follows that
\[
\left(1+\frac mn\right)^X
=
\left(1+\frac YX\right)^X
\ge
1+Y+\frac{X-1}{2X}Y^2 .
\]
We claim that
\[
1+Y+\frac{X-1}{2X}Y^2>2Y .
\]
Indeed, if \(Y=1\), then the left-hand side is strictly larger than \(2\).
If \(Y\ge 2\), then \(X\ge Y\) gives
\[
\frac{X-1}{2X}Y^2
\ge
\frac{Y-1}{2Y}Y^2
=
\frac{Y(Y-1)}2
\ge Y-1,
\]
and the claim follows. Hence
\[
\left(1+\frac mn\right)^X>2Y.
\]
On the other hand,
\[
\left(1+\frac ab\frac nm\right)^Y
>
Y\frac ab\frac nm
=
X\frac ab .
\]
Consequently,
\[
a\binom{(n+m)/a}{n/a,m/a}\Big/
        {\binom{(n+m)/b}{n/b,m/b}}
>
a\cdot 2Y\cdot X\frac ab
=
2anm\left(\frac1a-\frac1b\right)^2\frac ab .
\]
Since \(a=p^s\), \(b=q^t\), and
\[
n=p^\alpha q^\beta n',\qquad m=p^\gamma q^\delta m',
\]
we have
\[
2anm\left(\frac1a-\frac1b\right)^2\frac ab
=
2\Delta_{n,m}(a,b)q^\beta
\left(\frac{(b-a)^2pa^2}{b^2}\right).
\]
It remains to show that
\[
\frac{(b-a)^2pa^2}{b^2}\ge 1
\]
when \(\{a,b\}\ne\{2,3\}\). If \(b-a\ge 2\), then, since \(b<2a\),
\[
\frac{(b-a)^2pa^2}{b^2}
\ge
\frac{4pa^2}{b^2}
>
p
\ge 2.
\]
Thus we may assume \(b=a+1\). Since \(\{a,b\}\ne\{2,3\}\), we have
\(a\ge 3\), and therefore
\[
\frac{(b-a)^2pa^2}{b^2}
=
\frac{pa^2}{(a+1)^2}
=
\frac{p}{(1+1/a)^2}
\ge
\frac{9p}{16}
>1,
\]
as desired.

Now, we consider the case $\{a,b\}=\{2,3\}$.  Write
\[
  n=6N,\qquad m=6M,
\]
where \(N,M\ge 1\).
We need to prove
\[
  2\binom{3(N+M)}{3N}-(3^\beta-3)\binom{2(N+M)}{2N}
  >6\binom{2(N+M)}{2N},
\]
or equivalently
\begin{equation}\label{eq:exceptional-target}
  2\binom{3(N+M)}{3N}>(3^\beta+3)\binom{2(N+M)}{2N}.
\end{equation}
For fixed \(N\), define
\[
  R_M:=\frac{\binom{3(N+M)}{3N}}{\binom{2(N+M)}{2N}}.
\]
We claim that \(R_M\) is increasing in \(M\).  A direct calculation gives
\[
  \frac{R_{M+1}}{R_M}
  =
  \frac{(3N+3M+1)(3N+3M+2)(2M+1)}
       {(3M+1)(3M+2)(2N+2M+1)}.
\]
Now
\[
  \frac{3N+3M+1}{3M+1}
  =1+\frac{3N}{3M+1}
  >1+\frac{2N}{2M+1}
  =\frac{2N+2M+1}{2M+1},
\]
because \(3/(3M+1)>2/(2M+1)\).  Also
\(
  \frac{3N+3M+2}{3M+2}>1.
\)
Hence \(R_{M+1}>R_M\), so \(R_M\ge R_1\).  Therefore
\[
  2R_M\ge 2R_1
  =\frac{(3N+2)(3N+1)}{2N+1}.
\]
Moreover,
\[
  \frac{(3N+2)(3N+1)}{2N+1}-3(N+1)
  =\frac{3N^2-1}{2N+1}>0.
\]
Thus
\(
  2R_M>3(N+1).
\)
Since \(n=6N\), one has
\(
  3^\beta=3^{1+v_3(N)}\le 3N.
\)
Consequently, we have
\[
  3^\beta+3\le 3(N+1)<2R_M,
\]
multiplying by \(\binom{2(N+M)}{2N}\) gives \eqref{eq:exceptional-target}.

(ii) In this case, by Lemma \ref{lemma:Binom_ineq}, we have
$$
 \begin{array}{rcl}
 a\displaystyle\binom{\frac{n+m}{a}}{\frac{n}{a},\frac{m}{a}}\Big/\binom{\frac{n+m}{b}}{\frac{n}{b},\frac{m}{b}}&\geq& a (1+\frac{m}{n})^{n(\frac{1}{a}-\frac{1}{b})}(1+\frac{a}{b}\frac{n}{m})^{m(\frac{1}{a}-\frac{1}{b})}\\[4mm]
  &> & a m (\frac{1}{a}-\frac{1}{b})   n(\frac{1}{a}-\frac{1}{b})\frac{a}{b}\\[4mm]
 &= &  anm(\frac{1}{a}-\frac{1}{b})^2\frac{a}{b}\\[4mm]
 &= & \Delta_{n,m}(a,b) q^\beta   \left((b-a)^2 \frac{pa^2}{b^2}\right).
 \end{array}
 $$
Similar to the above, we have $(b-a)^2\frac{pa^2}{b^2}\geq 1$ and the desired result follows.
\end{proof}

\begin{lemma}\label{lemma:min_ord}
Let $G$ and $H$ be two finite abelian groups of order $n$. Let $m$ be a positive integer.  Denote
$$\begin{array}{c}
\mathcal{E}_G:=\{d\in \mathbb N\ |\ \varphi_G(d) > \varphi_H(d)\text{ for }d|(n,m)\},\\[2mm]
\mathcal{E}_H:= \{d\in \mathbb N\ |\ \varphi_G(d)< \varphi_H(d)\text{ for }d|(n,m)\}.
\end{array}$$
If $\mathcal{E}_G$ (resp. $\mathcal{E}_H$) is nonempty, then $\min{\mathcal{E}_G}$ (resp. $\min{\mathcal{E}_H}$) is a prime power.
\end{lemma}

\begin{proof} Let
$$n=\prod_{i=1}^s p_i^{n_i},\quad m=\prod_{i=1}^s p_i^{m_i},$$
where $n_i,m_i\geq 0$ for $1\le i\le s$.
First, we assume that $\mathcal{E}_G$ is nonempty.

For any $d|(n,m)$, let $d=\prod_{i=1}^s p_i^{d_i}$, where $0 \leq d_i\leq \min\{n_i,m_i\}$ for $1\le i\le s$.
Observe that
$$G=\operatorname{Syl}_{p_1}(G)\oplus\cdots\oplus \operatorname{Syl}_{p_{s}}(G).$$
Therefore,
$g=(g_1,g_2,\cdots,g_s)\in G$  (where $g_i\in  \operatorname{Syl}_{p_i}(G)$) has order $d$ if and only if ord$(g_i)=p_i^{d_i}$.
It follows that
$$\varphi_G(d)=\prod_{i=1}^s \varphi_{\operatorname{Syl}_{p_i}(G)}(p_i^{d_i}).$$  Consequently,  if  $\varphi_G(d) > \varphi_H(d)$, there exists some $i \in \{1, \dots, s\}$ such that
\[
\varphi_{\operatorname{Syl}_{p_i}(G)}(p_i^{d_i}) = \varphi_G(p_i^{d_i}) > \varphi_{\operatorname{Syl}_{p_i}(H)}(p_i^{d_i}) = \varphi_H(p_i^{d_i}).
\] The desired result follows immediately. The argument is analogous when $\mathcal{E}_H$ is nonempty.
\end{proof}

\begin{lemma}\label{lemma:min-order_p}
Let $G$ and $H$ be two finite abelian groups of order $n$. Let $m $ be a positive integer. Let $n=q^\beta n'$  and $m=q^\delta m'$ with $(n',q)=(m',q)=1$. Let
$$\begin{array}{c}
\mathcal{E}=\{k\in\mathbb{N} \ \big |\ \varphi_G(q^k) \neq \varphi_H(q^k),\ q^k|(n,m) \}.
\end{array}$$
Suppose that $\mathcal{E}$ is nonempty and let $t=\min\mathcal{E}$. If $\varphi_G(q^t) < \varphi_H(q^t),$ then we have $q^{t+1}|n$, i.e., $\beta>t$.
Moreover,
$$ q^t\leq  \varphi_H(q^t)-\varphi_G(q^t)  \leq  q^\beta -q^t.$$
\end{lemma}
\begin{proof}
Let
$$\operatorname{Syl}_q(G)= C_{q^{n_1}}\oplus\cdots\oplus C_{q^{n_c}},\quad
\operatorname{Syl}_q(H)=C_{q^{m_1}}\oplus\cdots\oplus C_{q^{m_d}},$$
where $1\leq n_1\leq \cdots \leq n_c$ and  $1\leq m_1\leq \cdots \leq m_d.$ Since $\varphi_G(q^t) < \varphi_H(q^t)$, it follows that $m_d \ge t$.  We claim that $d\geq 2.$ Otherwise, we also have $c=1$.  In other words, both $\operatorname{Syl}_q(G)$ and $\operatorname{Syl}_q(H)$ are cyclic groups of order at least $q^t$, which implies
$\varphi_G(q^t) = q^t - q^{t-1} = \varphi_H(q^t),$ contradicting our assumption. This proves the claim. Note that $d\geq 2$ implies $q^{m_1+m_d}| n,$ that is $q^{t+1}|n,$ i.e., $\beta>t$.

It is easy to verify that $\sum_{i=0}^t \varphi_G(q^i)=q^{e_1}$ and  $\sum_{j=0}^t \varphi_H(q^j)=q^{e_2}$ for some $e_1$ and $e_2$ with $ t\leq e_1<e_2 \leq  \beta.$  As $\varphi_H(q^i)=\varphi_G(q^i)$ for $i=0,\cdots,t-1,$ we have
$$ \varphi_H(q^t)-\varphi_G(q^t) =\sum_{i=0}^t\left(\varphi_H(q^i)-\varphi_G(q^i)\right)= \sum_{i=0}^t\varphi_H(q^i) - \sum_{j=0}^t\varphi_G(q^j)=q^{e_2}-q^{e_1}.$$
Observe that $q^t\leq q^{e_1}(q^{e_2-e_1}-1)=q^{e_2}-q^{e_1}\leq q^\beta - q^t.$ Hence, we conclude that
$$ q^t\leq  \varphi_H(q^t)-\varphi_G(q^t)  \leq  q^\beta -q^t.$$
This completes the proof.
\end{proof}

\begin{lemma}\label{lemma:p_order_estimation}
Let $G$ and $H$ be two finite abelian groups of order $n$. Let $m $ be a positive integer. Let $n=p^\alpha n'$  and $m=p^\gamma m'$ with $(n',p)=(m',p)=1$. Define
$$\begin{array}{c}
\mathcal{E}=\{k\in\mathbb{N} \ \big |\ \varphi_G(p^k) \neq \varphi_H(p^k),\ p^k|(n,m) \}.
\end{array}$$
Suppose that $\mathcal{E}$ is nonempty, and let $s = \min \mathcal{E}$. Assume that the following conditions hold:
\begin{enumerate}
    \item $s \ge 2$;
    \item $p^{s+1} \mid (n,m)$;
    \item $\varphi_G(p^s) > \varphi_H(p^s)$, but $\varphi_G(p^{s+1}) < \varphi_H(p^{s+1})$.
\end{enumerate}
Then it follows that $\alpha \ge s + 2.$
\end{lemma}
\begin{proof}
Let
$$\operatorname{Syl}_p(G)= C_{p^{n_1}}\oplus\cdots\oplus C_{p^{n_c}}\text{ and }
\operatorname{Syl}_p(H)=C_{p^{m_1}}\oplus\cdots\oplus C_{p^{m_d}},$$
where $1\leq n_1\leq \cdots \leq n_c$ and  $1\leq m_1\leq \cdots \leq m_d$. Following the argument in Lemma \ref{lemma:min-order_p}, we have $c\geq 2$ and $n_c\geq s$, because $\varphi_G(p^k)=\varphi_H(p^k)$ for $k<s$ and  $\varphi_G(p^s)>\varphi_H(p^s)$. Since $\varphi_G(p^{s+1}) < \varphi_H(p^{s+1})$,  we have  $m_d\geq s+1$. If $d=1,$ then $\varphi_H(p)=p-1.$ Therefore, we have $\varphi_G(p)= p^c-1\geq p^2-1>p-1=\varphi_H(p)$ and $s\leq 1$, which contradicts the assumption that $s\geq2$. Consequently, we have $d\geq 2$ and $p^{m_1+m_{d}}|n$, i.e., $\alpha\geq s+2$.
\end{proof}

The following proof splits according to whether the first negative discrepancy \(b\) satisfies
\(b\ge 2a\) or \(a<b<2a\). In the first case Lemma \ref{lemma:Binom_ineq} gives immediate domination.
In the second case Lemma \ref{lemma:Impartant_estimation} controls the first negative term, and the remaining
nearby negative terms are handled by a chaining argument through prime powers.

\begin{proof}[Proof of Lemma \ref{lem:rigidity-order-distribution}.]  Let $|G|=|H|=n$ and $m$ be a positive integer.  It suffices to prove that if $|\mathsf M(G,m)|=|\mathsf M(H,m)|$, then we have $\varphi_G(d)=\varphi_H(d)$ for any $d|(n,m)$. Assume to the contrary that there exists some $d|(n,m)$ such that $\varphi_G(d)\neq\varphi_H(d)$.

Recall that
$$|\mathsf M(G,m)|=\frac{1}{n+m}\sum_{d| (n,m)}\varphi_G(d)\binom{\frac{n+m}{d}}{\frac{n}{d}, \frac{m}{d}}$$
and
$$|\mathsf M(H,m)|=\frac{1}{n+m}\sum_{d| (n,m)}\varphi_{H}(d)\binom{\frac{n+m}{d}}{\frac{n}{d}, \frac{m}{d}}.$$

Let
$$a=\min\{d\ |\ \varphi_G(d)\neq \varphi_H(d)\text{ for }d\mid (n,m)\}.$$
Without loss of generality, we assume that
$\varphi_G(a)> \varphi_H(a).$
Under this assumption, we aim to show that
\begin{equation}\label{largereciprocity}
|\mathsf M(G,m)|>|\mathsf M(H,m)|,
\end{equation}
which contradicts our assumption.

It is clear that if $\varphi_G(d)\ge \varphi_H(d)$ holds for any $d|(n,m)$, then the desired result follows. Hence, we may assume that there exists a divisor $d$ of $(n,m)$ such that $\varphi_G(d)< \varphi_H(d)$.
Let
$$b=\min\{d\ |\ \varphi_G(d)< \varphi_H(d)\text{ for }d|(n,m)\}.$$
By definition, we have $a<b$.
Moreover, by Lemma \ref{lemma:min_ord}, both $a$ and $b$ are prime powers.
Recall that
$$\mathcal{E}_H=\{e\in \mathbb{N}\  \big  |\    \varphi_G(e)< \varphi_H(e) \text{ and } e|(n,m) \}.$$
We denote
$$\mathcal{F}:=\{e\in \mathbb{N}\  \big  |\    \varphi_G(e)> \varphi_H(e),\ e>a,    \text{ and } e|(n,m) \},$$
and
$$S_{\mathcal{F}}:=\sum_{
    e\in \mathcal{F}
 } \left(\varphi_G(e)-\varphi_H(e)\right)\binom{\frac{n+m}{e}}{\frac{n}{e}, \frac{m}{e}}.$$
It is clear that $S_{\mathcal{F}}\ge 0$. Therefore, we have
$$\begin{array}{rcl}
& &(n+m)\displaystyle\left(|\mathsf M(G,m)|-|\mathsf M(H,m)|\right) \\[4mm]
&=& \displaystyle\sum_{d| (n,m)} \left(\varphi_G(d)-\varphi_H(d)\right) \binom{\frac{n+m}{d}}{\frac{n}{d}, \frac{m}{d}}\\[4mm]
&= &  \displaystyle  \left(\varphi_G(a)-\varphi_H(a)\right)\binom{\frac{n+m}{a}}{\frac{n}{a},\frac{m}{a}} +S_{\mathcal{F}} +\sum_{
    e\in \mathcal{E}_H
 } \left(\varphi_G(e)-\varphi_H(e)\right)\binom{\frac{n+m}{e}}{\frac{n}{e}, \frac{m}{e}} \\[4mm]
 &= & \displaystyle  \left(\varphi_G(a)-\varphi_H(a)\right)
 \binom{\frac{n+m}{a}}{\frac{n}{a},\frac{m}{a}}+S_{\mathcal{F}} -\sum_{
    e\in \mathcal{E}_H
 }(\varphi_H(e)- \varphi_G(e))\binom{\frac{n+m}{e}}{\frac{n}{e}, \frac{m}{e}} \\[4mm]
  &\geq & \displaystyle a\binom{\frac{n+m}{a}}{\frac{n}{a},\frac{m}{a}}+S_{\mathcal{F}}-\sum_{
    e\in \mathcal{E}_H
 } (\varphi_H(e)- \varphi_G(e))\binom{\frac{n+m}{e}}{\frac{n}{e}, \frac{m}{e}},\\
\end{array}
$$
where the last inequality follows from Lemma \ref{lemma:min-order_p}.

Therefore, in order to prove (\ref{largereciprocity}), it suffices to show that
\begin{equation}\label{sufficestoprove}
a\binom{\frac{n+m}{a}}{\frac{n}{a},\frac{m}{a}}+S_{\mathcal{F}}  > \sum_{
    e\in \mathcal{E}_H
 } (\varphi_H(e)- \varphi_G(e))\binom{\frac{n+m}{e}}{\frac{n}{e}, \frac{m}{e}}.
\end{equation}
As $a$ and $b$ are prime powers, we may assume that $a=p^s$ and $b=q^t$, where $p$ and $q$ are primes (not necessarily distinct).

If $b\ge 2a$, i.e., $q^t\geq 2 p^s$. By Lemma \ref{lemma:Binom_ineq}.(ii), we have
$$a\binom{\frac{n+m}{a}}{\frac{n}{a},\frac{m}{a}}  > n \binom{\frac{n+m}{b}}{\frac{n}{b}, \frac{m}{b}}\geq \sum_{
    e\in \mathcal{E}_H
 } \varphi_H(e)\binom{\frac{n+m}{b}}{\frac{n}{b}, \frac{m}{b}}\geq \sum_{
    e\in \mathcal{E}_H
 } (\varphi_H(e)- \varphi_G(e))\binom{\frac{n+m}{e}}{\frac{n}{e}, \frac{m}{e}},$$
where the last inequality follows from Lemma \ref{lemma:Binom_ineq}.(i). Thus, (\ref{sufficestoprove}) follows immediately.

If $a<b< 2a$, i.e., $p^s<q^t< 2 p^s$. In this case, we have $$q^{t-1}=\frac{q^t}{q}< \frac{2p^s}{q}\leq p^s.$$
Consequently, we have $p\neq q$. Moreover, by the definition of $b=q^t$, we have $t=\min\{ i\in\mathbb N \ \big | \ \varphi_G(q^i)\neq  \varphi_H(q^i)\}$ and  $\varphi_G(q^t)<  \varphi_H(q^t)$. By Lemma \ref{lemma:min-order_p}, we have $q^{t+1}|n$ and $p^{s+1}|n$.  Assume that
$$n=p^\alpha q^\beta n',\quad m=p^\gamma q^\delta m',$$
where $(n',pq)=(m',pq)=1$.
Then we have  $a=p^s<b=q^t<2p^s$  and  $1\leq s< \alpha,$ $1\leq t< \beta$, and  $s\leq \gamma$, $t\leq \delta$. Consequently, we have $pq|n(\frac{1}{a}-\frac{1}{b})$, which implies $n(\frac{1}{a}-\frac{1}{b})\geq pq\geq 6$.

Recall that $\Delta_{n,m}(a,b)=p^{\alpha+\gamma-2s-1}q^{\delta-t}n'm'$.  As $\alpha+\gamma>2s,$ $\Delta_{n,m}(a,b)\geq 1.$ By Lemma \ref{lemma:Impartant_estimation}.(i),
 we have
\begin{equation}\label{Case2_a2b}
 a\binom{\frac{n+m}{a}}{\frac{n}{a},\frac{m}{a}}- (q^\beta-q^{t} )  \binom{\frac{n+m}{b}}{\frac{n}{b},\frac{m}{b}}>  2b  \binom{\frac{n+m}{b}}{\frac{n}{b},\frac{m}{b}}.
\end{equation}
We denote
$$\mathcal{E}_1:=\{e\in\mathcal{E}_H\ \big | \ b<e<2b\},\quad  \mathcal{E}_2:=\{e\in\mathcal{E}_H\ \big | \ e\geq 2b\}.$$ By Lemma \ref{lemma:min-order_p}, we have  $\varphi_H(b)- \varphi_G(b)\leq q^\beta -q^{t}.$ Therefore, by (\ref{Case2_a2b}), we obtain
\begin{align*}
   &\quad\ S_{\mathcal F}+\displaystyle a\binom{\frac{n+m}{a}}{\frac{n}{a},\frac{m}{a}} - \sum_{
    e\in \mathcal{E}_H
 } (\varphi_H(e)- \varphi_G(e))\binom{\frac{n+m}{e}}{\frac{n}{e}, \frac{m}{e}} \\
 &\geq S_{\mathcal F}+\displaystyle a\binom{\frac{n+m}{a}}{\frac{n}{a},\frac{m}{a}} - (\varphi_H(b)- \varphi_G(b))\binom{\frac{n+m}{b}}{\frac{n}{b}, \frac{m}{b}} -\sum_{
    e\in \mathcal{E}_1\cup  \mathcal{E}_2
 } \varphi_H(e)\binom{\frac{n+m}{e}}{\frac{n}{e}, \frac{m}{e}}\\
  &>  \displaystyle S_{\mathcal F}+ 2b \binom{\frac{n+m}{b}}{\frac{n}{b}, \frac{m}{b}} -\sum_{
    e_1\in \mathcal{E}_1
 } \varphi_H(e_1)\binom{\frac{n+m}{e_1}}{\frac{n}{e_1}, \frac{m}{e_1}}-\sum_{
    e_2\in \mathcal{E}_2
 } \varphi_H(e_2)\binom{\frac{n+m}{e_2}}{\frac{n}{e_2}, \frac{m}{e_2}}\\
   &=  \displaystyle S_{\mathcal F}+b \binom{\frac{n+m}{b}}{\frac{n}{b}, \frac{m}{b}} -\sum_{
    e_1\in \mathcal{E}_1
 } \varphi_H(e_1)\binom{\frac{n+m}{e_1}}{\frac{n}{e_1}, \frac{m}{e_1}}+b \binom{\frac{n+m}{b}}{\frac{n}{b}, \frac{m}{b}} -\sum_{
    e_2\in \mathcal{E}_2
 } \varphi_H(e_2)\binom{\frac{n+m}{e_2}}{\frac{n}{e_2}, \frac{m}{e_2}}.
\end{align*}
Denote
$$\mathcal S_1:=S_{\mathcal F}+b \binom{\frac{n+m}{b}}{\frac{n}{b}, \frac{m}{b}} -\sum_{
    e_1\in \mathcal{E}_1
 } \varphi_H(e_1)\binom{\frac{n+m}{e_1}}{\frac{n}{e_1}, \frac{m}{e_1}}$$
and
$$
 \mathcal S_2:=b \binom{\frac{n+m}{b}}{\frac{n}{b} \frac{m}{b}} -\sum_{
    e_2\in \mathcal{E}_2
 } \varphi_H(e_2)\binom{\frac{n+m}{e_2}}{\frac{n}{e_2}, \frac{m}{e_2}}.$$
In order to prove (\ref{sufficestoprove}), it suffices to show that $\mathcal S_1>0$ and $\mathcal S_2>0$.

First, we consider $\mathcal S_2$. If $\mathcal{E}_2$  is empty, then the desired result follows immediately. Otherwise, assume that $\mathcal{E}_2$  is not empty, let $c=\min \mathcal{E}_2.$ By definition, we have $c\geq 2b$. Therefore, by Lemma \ref{lemma:Binom_ineq}.(ii), we have
$$
\begin{array}{rcl}
\displaystyle b \binom{\frac{n+m}{b}}{\frac{n}{b}, \frac{m}{b}} -\sum_{
    e_2\in \mathcal{E}_2
 } \varphi_H(e_2)\binom{\frac{n+m}{e_2}}{\frac{n}{e_2}, \frac{m}{e_2}} &\geq& \displaystyle  b \binom{\frac{n+m}{b}}{\frac{n}{b}, \frac{m}{b}} -\sum_{
    e_2\in \mathcal{E}_2
 } \varphi_H(e_2)\binom{\frac{n+m}{c}}{\frac{n}{c}, \frac{m}{c}}\\[5mm]
 &\geq & \displaystyle     b \binom{\frac{n+m}{b}}{\frac{n}{b}, \frac{m}{b}} - n \binom{\frac{n+m}{c}}{\frac{n}{c}, \frac{m}{c}}>0.
\end{array}
$$
Therefore, $\mathcal S_2>0$, as desired.

Next, we consider $\mathcal S_1$. If $\mathcal{E}_1$ is empty, then the desired result follows. So, we may assume that $\mathcal{E}_1$ is not empty.
We claim that every element of $\mathcal{E}_1$ is a prime power, and the primes involved are distinct. More precisely, we can write $\mathcal{E}_1 = \{q_1^{t_1}, q_2^{t_2}, \dots, q_v^{t_v}\}, \:\text{where} \: q_i \neq q_j \ \text{for } i \neq j.$ To see this, take any $e_1 \in \mathcal{E}_1$. Since $\varphi_G(e_1) < \varphi_H(e_1)$ and $q^t < e_1 < 2 q^t$, we can factor $e_1$ as $e_1 = \ell^k e_1',$ with $\ell$ a prime and $e_1' \in \mathbb{N}$, where $\varphi_G(\ell^k) < \varphi_H(\ell^k)$.
If $e_1' \ge 2$, then $\ell^k = \frac{e_1}{e_1'} < \frac{2 q^t}{2} = q^t,$ which contradicts the minimality of $b = q^t$. Therefore, $e_1' = 1$, and each $e_1$ is indeed a prime power.


Now denote $$\mathcal{E}_1^{G}:=\{\ell^k\in \mathcal{E}_1  \ \big |\ \ell\neq p\   \text{and }\varphi_G(\ell^i)> \varphi_H(\ell^i)\text{ for some }i<k \}.$$

For any $\ell^k\in\mathcal{E}_1^{G}$, let $u=\min\{i\in\mathbb N\ |\ \varphi_G(\ell^i)> \varphi_H(\ell^i)\}$. Moreover, by the definitions of $b$ and $\mathcal{E}_1$, we also have
$$u=\min\{i\in\mathbb N\ |\ \varphi_G(\ell^i)\neq \varphi_H(\ell^i)\}.$$
Using Lemma \ref{lemma:min-order_p}, we have $\varphi_G(\ell^u)-\varphi_H(\ell^u)\ge \ell^u$. By Lemma \ref{lemma:Binom_ineq}.(ii), we have
\begin{equation}\label{sf_larger_than_e1}
(\varphi_G(\ell^u)-\varphi_H(\ell^u))\binom{\frac{n+m}{\ell^u}}{\frac{n}{\ell^u}, \frac{m}{\ell^u}}
\ge \ell^u\binom{\frac{n+m}{\ell^u}}{\frac{n}{\ell^u}, \frac{m}{\ell^u}}
>\varphi_H(\ell^k)\binom{\frac{n+m}{\ell^k}}{\frac{n}{\ell^k}, \frac{m}{\ell^k}}.
\end{equation}
Recall that
$S_{\mathcal{F}}=\sum_{
    e\in \mathcal{F}
 } \left(\varphi_G(e)-\varphi_H(e)\right)\binom{\frac{n+m}{e}}{\frac{n}{e}, \frac{m}{e}}$.
Since $\mathcal{E}_1^{G}$ consists of powers of distinct primes (which are different from $p,q$), by (\ref{sf_larger_than_e1}), we obtain
$$
\displaystyle \begin{array}{rcl}
&&\displaystyle S_{\mathcal F}+b \binom{\frac{n+m}{b}}{\frac{n}{b}, \frac{m}{b}} -\sum_{
    e_1\in \mathcal{E}_1
 } \varphi_H(e_1)\binom{\frac{n+m}{e_1}}{\frac{n}{e_1}, \frac{m}{e_1}}\\[4mm]
 &= &\displaystyle S_{\mathcal F}-\sum_{
    e\in \mathcal{E}_1^{G}
 } \varphi_H(e)\binom{\frac{n+m}{e}}{\frac{n}{e}, \frac{m}{e}}+\displaystyle b \binom{\frac{n+m}{b}}{\frac{n}{b}, \frac{m}{b}} -\sum_{
    e_1\in \mathcal{E}_1\setminus \mathcal{E}_1^{G}
 } \varphi_H(e_1)\binom{\frac{n+m}{e_1}}{\frac{n}{e_1}, \frac{m}{e_1}}\\[5mm]
 &\geq &\displaystyle b \binom{\frac{n+m}{b}}{\frac{n}{b}, \frac{m}{b}} -\sum_{
    e_1\in \mathcal{E}_1\setminus \mathcal{E}_1^{G}
 } \varphi_H(e_1)\binom{\frac{n+m}{e_1}}{\frac{n}{e_1}, \frac{m}{e_1}} .
 \end{array}
 $$
As a result, to prove $\mathcal S_1>0$, it suffices to show that
\begin{equation}\label{sufficestoprove_S2}
b \binom{\frac{n+m}{b}}{\frac{n}{b}, \frac{m}{b}} -\sum_{
    e_1\in \mathcal{E}_1\setminus \mathcal{E}_1^{G}
 } \varphi_H(e_1)\binom{\frac{n+m}{e_1}}{\frac{n}{e_1}, \frac{m}{e_1}}>0.
\end{equation}

If $\mathcal{E}_{1}\setminus \mathcal{E}_1^{G}$ is empty, then we have $\mathcal S_1>0$. Therefore, suppose that $\mathcal{E}_{1}\setminus \mathcal{E}_1^{G}$ is nonempty. Without loss of generality, we may assume that
$$\mathcal{E}_{1}\setminus \mathcal{E}_1^{G}=\{q_1^{t_1},\cdots,q_L^{t_L}\},$$
where $L\le v$ and $q_1^{t_1}<\cdots< q_L^{t_L}$. We denote $b_0:=b$ (with $q_0:=q$ and $t_0:=t$) and $b_i:=q_i^{t_i}$ for $i=1,2,\cdots,L$. Therefore, $b_0<b_1<b_2<\cdots<b_L<2b=2q^t$. For each $i=0,1,\cdots, L-1,$ let
$$n=q_i^{\beta_i}q_{i+1}^{\beta_{i+1}} n_i,\quad m=q_i^{\delta_i}q_{i+1}^{\delta_{i+1}} m_i,$$ where $(n_i,q_iq_{i+1})=(m_i,q_iq_{i+1})=1,$ and denote
$$\Delta_{n,m}(b_i,b_{i+1}):= q_i^{\beta_i+\delta_{i}-2t_i-1}q_{i+1}^{\delta_{i+1}-t_{i+1}}n_im_i.$$

In the following, we shall prove that
\begin{equation}\label{eq:summand_term_i}
    \displaystyle b_i \binom{\frac{n+m}{b_i}}{\frac{n}{b_i}, \frac{m}{b_i}} -\varphi_H(b_{i+1})\binom{\frac{n+m}{b_{i+1}}}{\frac{n}{b_{i+1}}, \frac{m}{b_{i+1}}}> b_{i+1}\binom{\frac{n+m}{b_{i+1}}}{\frac{n}{b_{i+1}}, \frac{m}{b_{i+1}}},\quad i=0, 1,\cdots,L-1.
\end{equation}
Note that, if (\ref{eq:summand_term_i}) holds, then we have
$$\sum_{i=0}^{L-1}\left(\displaystyle b_i \binom{\frac{n+m}{b_i}}{\frac{n}{b_i}, \frac{m}{b_i}} -\varphi_H(b_{i+1})\binom{\frac{n+m}{b_{i+1}}}{\frac{n}{b_{i+1}}, \frac{m}{b_{i+1}}}\right)>\sum_{i=0}^{L-1} b_{i+1}\binom{\frac{n+m}{b_{i+1}}}{\frac{n}{b_{i+1}}, \frac{m}{b_{i+1}}},
$$
which implies
$$b_0 \binom{\frac{n+m}{b_0}}{\frac{n}{b_0}, \frac{m}{b_0}}-\sum_{i=0}^{L-1} \varphi_H(b_{i+1})\binom{\frac{n+m}{b_{i+1}}}{\frac{n}{b_{i+1}}, \frac{m}{b_{i+1}}}>b_L \binom{\frac{n+m}{b_L}}{\frac{n}{b_L}, \frac{m}{b_L}}.$$
Therefore, we have
$$b \binom{\frac{n+m}{b}}{\frac{n}{b}, \frac{m}{b}} -\sum_{
    e_1\in \mathcal{E}_1\setminus \mathcal{E}_1^{G}
 } \varphi_H(e_1)\binom{\frac{n+m}{e_1}}{\frac{n}{e_1}, \frac{m}{e_1}} >b_L \binom{\frac{n+m}{b_L}}{\frac{n}{b_L}, \frac{m}{b_L}}>0,$$
and (\ref{sufficestoprove_S2}) follows.

Now, we prove (\ref{eq:summand_term_i}). Note that, for each $i=0,1,\cdots,L-1$,
\begin{itemize}

\item if $q_i=p$, then we have  $q_i^{t_i}=p^{s+1}$, as $q_i^{t_i}<2b<4p^s$;

\item if $q_i\neq p$, by Lemma \ref{lemma:min-order_p} (note that $t_i=\min\{j\ \big |\ \varphi_G(q_i^j)\neq \varphi_H(q_i^j)\}$ and that $\varphi_G(q_i^{t_i})< \varphi_H(q_i^{t_i})$ as $b_i=q_i^{t_i}\in \mathcal{E}_1\setminus \mathcal{E}_1^G$),  we have $\beta_i>t_i$;

\item we have $\{b_i,b_{i+1}\}\neq \{2,3\}$, as $b_i>q^t\geq 2$.

\end{itemize}
We distinguish the following three cases.

{\bf Case 1:} Assume that $b_{i+1}\neq p^{s+1}$.  We first show that $n\left(\frac{1}{b_i} - \frac{1}{b_{i+1}}\right) \ge 3.$
Indeed, if $n$ has at least three distinct prime divisors, then the above inequality holds.

If $b_i\neq p^{s+1}$, as mentioned above, we have $\beta_i>t_i$. Therefore, $p$, $q_i$, and $q_{i+1}$ are different prime divisors of $n$.

If $b_i= p^{s+1},$  we have $q$, $p$, and $q_{i+1}$ are   different prime divisors  of $n.$ Moreover, we have $q^t|n_i$ and  $\frac{n_i}{q_i}\geq \frac{q^t}{q_i} \geq \frac{q^t}{p^s}>1$, where $q_i=p$.

Therefore, in all cases, we have $n(\frac{1}{b_i}-\frac{1}{b_{i+1}})\geq 3$ and  $\Delta_{n,m}(b_i,b_{i+1})\geq 1$. By Lemma \ref{lemma:Impartant_estimation}.(i), it follows that
    $$\begin{array}{rcl}
     \displaystyle b_i \binom{\frac{n+m}{b_i}}{\frac{n}{b_i}, \frac{m}{b_i}}    >2 \Delta_{n,m}(b_{i},b_{i+1}) q_{i+1}^{\beta_{i+1}}\binom{\frac{n+m}{b_{i+1}}}{\frac{n}{b_{i+1}}, \frac{m}{b_{i+1}}}\geq 2 q_{i+1}^{\beta_{i+1}}\binom{\frac{n+m}{b_{i+1}}}{\frac{n}{b_{i+1}}, \frac{m}{b_{i+1}}}.
\end{array}
$$
Since $b_{i+1}< q_{i+1}^{\beta_{i+1}}$ and $\varphi_H(b_{i+1})< q_{i+1}^{\beta_{i+1}}$, \eqref{eq:summand_term_i} follows.

{\bf Case 2:}  Assume that $b_{i+1}=p^{s+1}$ and $s\geq 2$. Then \(q_{i+1}=p\) and \(t_{i+1}=s+1\).  By Lemma \ref{lemma:p_order_estimation}, applied to the sign change
\(
  \varphi_G(p^s)>\varphi_H(p^s),\
  \varphi_G(p^{s+1})<\varphi_H(p^{s+1}),
\)
we have
\begin{equation}\label{eq:beta-p-lower}
  \beta_{i+1}=v_p(n)\ge s+2.
\end{equation}
Moreover, \(q_i\ne p\).  Indeed, we have \(
  p^s=a<b_i<b_{i+1}=p^{s+1}.
\)

By the definition of \(b\), and by the construction of
\(\mathcal E_1\setminus \mathcal E_1^G\), the discrepancy at
\(b_i=q_i^{t_i}\) is the first discrepancy for the prime \(q_i\), and it is negative.
Hence Lemma \ref{lemma:min-order_p} gives
\begin{equation}\label{eq:beta-i-lower}
  \beta_i=v_{q_i}(n)>t_i.
\end{equation}
Also, since both \(b_i\) and \(b_{i+1}\) divide \((n,m)\), we have
\begin{equation}\label{eq:delta-lower}
  \delta_i=v_{q_i}(m)\ge t_i,
  \qquad
  \delta_{i+1}=v_p(m)\ge s+1.
\end{equation}
Write
\[
  n=q_i^{\beta_i}p^{\beta_{i+1}}n_i,
  \qquad
  m=q_i^{\delta_i}p^{\delta_{i+1}}m_i,
\]
where \((n_i,q_ip)=(m_i,q_ip)=1\).  Then
\[
  n\left(\frac1{b_i}-\frac1{b_{i+1}}\right)
  =n\left(\frac1{q_i^{t_i}}-\frac1{p^{s+1}}\right)
  =p^{\beta_{i+1}-s-1}q_i^{\beta_i-t_i}n_i
    \left(p^{s+1}-q_i^{t_i}\right).
\]
The last factor is a positive integer because \(b_i<b_{i+1}\).  By
\eqref{eq:beta-p-lower}, \eqref{eq:beta-i-lower}, and \(q_i\ne p\), the right-hand side is at
least
\(
  p q_i\ge 6.
\)
In particular,
\begin{equation}\label{eq:n-gap-ge3}
  n\left(\frac1{b_i}-\frac1{b_{i+1}}\right)
  \ge 3.
\end{equation}
Furthermore, with the notation of Lemma \ref{lemma:Impartant_estimation} applied to
\(a=b_i=q_i^{t_i}\) and \(b=b_{i+1}=p^{s+1}\), we have
\begin{equation}\label{eq:Delta-bi-bip1}
  \Delta_{n,m}(b_i,b_{i+1})
  =q_i^{\beta_i+\delta_i-2t_i-1}
   p^{\delta_{i+1}-s-1}n_i m_i.
\end{equation}
The exponents in \eqref{eq:Delta-bi-bip1} are nonnegative, since
\[
  \beta_i\ge t_i+1,
  \qquad
  \delta_i\ge t_i,
  \qquad
  \delta_{i+1}\ge s+1.
\]
Therefore
\begin{equation}\label{eq:Delta-ge1}
  \Delta_{n,m}(b_i,b_{i+1})\ge 1.
\end{equation}
Also \(\{b_i,b_{i+1}\}\ne \{2,3\}\), because \(s\ge 2\) and \(b_i>a=p^s\).  Hence Lemma \ref{lemma:Impartant_estimation}.(i), together with
\eqref{eq:n-gap-ge3} and \eqref{eq:Delta-ge1}, gives
\begin{align*}
  b_i\binom{(n+m)/b_{i}}{n/b_{i}}
  &>2\Delta_{n,m}(b_i,b_{i+1})p^{\beta_{i+1}}
    \binom{(n+m)/b_{i+1}}{n/b_{i+1}} \\
  &\ge 2p^{\beta_{i+1}}\binom{(n+m)/b_{i+1}}{n/b_{i+1}}.
\end{align*}
On the other hand,
\(
  b_{i+1}=p^{s+1}<p^{\beta_{i+1}}
\)
by \eqref{eq:beta-p-lower}, and trivially
\(
  \varphi_H(b_{i+1})<p^{\beta_{i+1}}
\).  Thus
\[
2p^{\beta_{i+1}}>b_{i+1}+\varphi_H(b_{i+1}).
\]
It follows that
\[
  b_i\binom{(n+m)/b_{i}}{n/b_{i}}
  >\bigl(b_{i+1}+\varphi_H(b_{i+1})\bigr)
   \binom{(n+m)/b_{i+1}}{n/b_{i+1}},
\]
which is exactly \eqref{eq:summand_term_i}.

{\bf Case 3:} Assume that $b_{i+1}=p^{s+1}$ and $s=1$. In this case, we have $p\leq 3$. In fact, as $b<2p^s \le p^{s+1}<2b$, we have $p/2=p^{s+1}/2p^s<2b/b=2$.

{\bf Subcase 3.1:} Assume that $p=2$. Then $a = p^s = 2$ and $b = q^t = 3$ and hence $\mathcal{E}_{1}\setminus \mathcal{E}_1^{G}\subseteq \{4,5\}$. Therefore, it suffices to consider the case  $b_{0}=q^{t_0}=3$ and $b_{1}=p^{s+1}=4.$ Note that $n(\frac{1}{a}-\frac{1}{b})\geq 6,$ which implies that $n\geq 36$.  Consequently, $n(\frac{1}{b_0}-\frac{1}{b_{1}})\geq \frac{36}{12}=3$.
As $\Delta_{n,m}(b_0,b_{1})\geq 1$, by Lemma \ref{lemma:Impartant_estimation}.(i),
    $$\begin{array}{rcl}
     \displaystyle b_{0} \binom{\frac{n+m}{b_0}}{\frac{n}{b_0}, \frac{m}{b_0}}    >2\Delta_{n,m}(b_{0},b_{1}) 2^{\beta_{1}}\binom{\frac{n+m}{b_{1}}}{\frac{n}{b_{1}}, \frac{m}{b_{1}}}\geq 2\cdot 2^{\beta_{1}}\binom{\frac{n+m}{b_{1}}}{\frac{n}{b_{1}}, \frac{m}{b_{1}}}.
\end{array}
$$
Since  $b_{1}< 2^{\beta_{1}}$ and $\varphi_H(b_{1})< 2^{\beta _{1}}$, \eqref{eq:summand_term_i} follows.


{\bf Subcase 3.2:} Assume that $p=3$. In this case, $q^t=5$. As a result, we have $a=p^s=3$ and $b=q^t=5$, and $\mathcal{E}_{1}\setminus \mathcal{E}_1^G\subseteq \{ 7,8,9\}.$ It suffices to consider the following three cases
$$
(1):\ \{b_0,b_{1}\}=\{5,9\},\quad (2):\ \{b_0,b_{1}\}=\{7,9\},\quad (3):\ \{b_0,b_{1}\}=\{8,9\}.$$
Since $\beta_0 > t_0$, we have $\Delta_{n,m}(b_0,b_1) \ge 1$.  In case (1), as $5 \mid n$ and $9 \mid n$, we obtain $n(\frac{1}{5}-\frac{1}{9})\geq 4$ . For the cases (2) and (3), the integer $n$ has at least three different prime divisors. Therefore, we always have $n(\frac{1}{b_0}-\frac{1}{b_{1}})\geq 3$. By Lemma \ref{lemma:Impartant_estimation}.(i),
$$\begin{array}{rcl}
     \displaystyle b_{0} \binom{\frac{n+m}{b_0}}{\frac{n}{b_0}, \frac{m}{b_0}}    >2\Delta_{n,m}(b_{0},b_{1}) 3^{\beta_{1}}\binom{\frac{n+m}{b_{1}}}{\frac{n}{b_{1}}, \frac{m}{b_{1}}}\geq 2\cdot 3^{\beta_{1}}\binom{\frac{n+m}{b_{1}}}{\frac{n}{b_{1}}, \frac{m}{b_{1}}}.
\end{array}
$$
Since $b_{1}< 3^{\beta_{1}}$ and $\varphi_H(b_{1})< 3^{\beta_{1}}$, \eqref{eq:summand_term_i} follows. This completes the proof.
\end{proof}

\section*{Acknowledgments}
We are grateful to Prof. Shishuo Fu and Prof. Alfred Geroldinger for very helpful suggestions and comments. This work was supported by Guangdong Basic and Applied Basic Research Foundation Grant No.~2024A1515012564 and the National Natural Science Foundation of China Grant No.~12131011.

\medskip
\begingroup\small
\noindent Department of Mathematics, Southwest Jiaotong University, Chengdu 610000, P.R. China

\noindent \emph{E-mail address:} handongchun@swjtu.edu.cn

\smallskip
\noindent School of Mathematics (Zhuhai), Sun Yat-sen University, Zhuhai 519082, Guangdong, P.R. China

\noindent \emph{E-mail addresses:} wangx728@mail2.sysu.edu.cn, zhanghb68@mail.sysu.edu.cn

\smallskip
\noindent Department of Mathematics, Southern University of Science and Technology, Shenzhen 518055, Guangdong, P.R. China

\noindent \emph{E-mail address:} 12431019@mail.sustech.edu.cn
\endgroup

\end{document}